\documentclass[12pt,a4paper]{article}
\usepackage{color}
\usepackage{xcolor}
\usepackage{graphicx}
\usepackage{tabularx}
\usepackage{amsmath,amssymb,mathrsfs,wasysym}
\usepackage{amsthm}
\usepackage{multicol}
\usepackage{enumerate}
\usepackage{float}
\usepackage{makeidx}  
\date{}

\newtheorem{theorem}{Theorem}

\newtheorem{prop}{Proposition}
\newtheorem{cor}{Corollary}
\theoremstyle{definition}
\newtheorem{defn}{Definition}
\newtheorem{ex}{Example}
\newtheorem{rem}{Remark}

\begin{document}
\title{\bf Affine and Projective Planes Linked with Projective Lines over Certain
Rings of Lower Triangular Matrices}
\author{\bf Edyta Bartnicka and Metod Saniga}
\maketitle
\begin{abstract}
\noindent
Let $T_n(q)$ be the ring of lower triangular  matrices of order $n \geq 2$ with entries from the finite field $F(q)$ of order $q \geq 2$ and
let ${^2T_n(q)}$ denote its free left module. For $n=2,3$ it is shown that the projective line over $T_n(q)$ gives rise to a set of $(q+1)^{(n-1)}q^{\frac{3(n-1)(n-2)}{2}}$ affine planes of order $q$. The points of such an affine plane are non-free cyclic submodules of ${^2T_n(q)}$ not contained in any non-unimodular free cyclic submodule of ${^2T_n(q)}$ and its lines are points of the projective line. 
Furthermore, it is demonstrated that each affine plane can be extended to the projective plane of order $q$, with the `line at infinity' being represented by those free cyclic submodules of ${^2T_n(q)}$ that are generated by non-unimodular pairs. Our approach can straightforwardly be adjusted to address the case of arbitrary $n$.
\end{abstract}

\section{Introduction}
The main motivation for this paper stems from physics. Some decade ago the second author and his colleagues found out that projective lines over finite rings have interesting applications in physics, see e.\,g. \cite{spp,havsa}. When dealing with  lines over small rings of order less than 32, it was noticed that in some cases such lines are not only naturally linked with (or give rise to) affine planes over certain finite fields, but the corresponding rings also feature free cyclic submodules that are generated by non-unimodular (or, what in the finite case amounts to the same, non-admissible) vectors \cite{sp}. This is done in the framework of projective lines over rings of specific lower triangular matrices, which not only incorporates the most relevant physical cases, but also clarifies the so-far unnoticed role played by non-unimodular free cyclic submodules. This paper is aimed at the rigorous mathematical treatment, as well as a deeper understanding, of  the above-described `observations/findings' of physicists. This is done in the framework of projective lines over rings of specific lower triangular matrices, which also incorporates the most relevant physical cases.

\section{Preliminaries}

\bigskip
Let $T_n(q)$ be the ring of lower triangular  matrices $$X=\left[\begin{array}{ccclr}
x_{11}&0&0&0\\x_{21}&x_{22}&0&0\\ \vdots& \vdots&\ddots&0\\ x_{n1}&x_{n2}&\dots&x_{nn}\end{array}\right]$$ of order $n \geq 2$ with entries from the finite field $F(q)$ of order $q \geq 2$. For $n=2$, i.e., in the case of the ring of ternions, we will use, as usual, the symbol  $T(q)$ instead of $T_2(q)$. Let us put $F^*(q) := F(q)\backslash\{0\}$ and let $T_n^*(q)$ stands for the group of invertible elements of the ring $T_n(q)$. The Jacobson radical of the ring will be denoted by $J$. Consider the free left module ${^2T_n(q)}$ over $T_n(q)$. If $\big(X, Y\big)\subset{^2T_n(q)}$ then the set: $T_n(q)\big(X, Y\big)=\{A\big(X, Y\big);A\in T_n(q)\}$ is a left cyclic submodule of ${^2T_n(q)}$. If the equation $(AX, AY)=(0, 0)$ implies that $A=0$, then $T_n(q)\big(X, Y\big)$ is free. 
\begin{rem}\label{generators.of.fcs}
Obviously, a submodule $T_n(q)\big(X, Y\big)\subset{^2T_n(q)}$ is free if, and only if, its order is $q^{S_n}$, where $S_n=1+2+\dots+n$. 

Additionally, such free cyclic submodule is generated exactly by $|T_n^*(q)|$ elements of the form  $U(X, Y)$, where $U$ runs through all the elements of $T_n^*(q)$. 
\end{rem}

A pair $(X, Y)\subset{^2T_n(q)}$ is {\sl unimodular}, if $x_{ii}\neq 0 \vee y_{ii}\neq 0$ for any $i=1, \dots, n$.
{\sl The projective line}  $\mathbb{P}(T_n(q))$ over the ring $T_n(q)$ is the set of all cyclic submodules $T_n(q)\big(X, Y\big)$, where $\big(X, Y\big)\subset{^2T_n(q)}$ is unimodular. 
We refer to \cite{her, design} for definitions of the unimodularity and the projective line in the case of an arbitrary finite associative ring with unity.

It is well known that any unimodular pair generates a free cyclic submodule. Therefore the projective line  $\mathbb{P}(T_n(q))$ is the set of those free cyclic submodules of ${^2T_n(q)}$ which are generated by unimodular pairs. They are called {\sl unimodular free cyclic submodules} or {\sl points} of $\mathbb{P}(T_n(q))$. 

\vspace*{0.3cm}

Let us recall some of the most elementary facts of ordinary plane geometry.

\begin{defn}\label{aff}
{\sl An affine plane}  is a system of points and lines with an incidence relation between the points and lines that satisfy the following axioms.

{\bf A1} Any two distinct points lie on a unique line.

{\bf A2} Given any line $l$ and any point $P$ not on $l$ there is a unique line $m$ which contains the point $P$ and does not meet the line $l$.

{\bf A3} There exist three non-collinear points. (A set of points $P_1, \dots , P_n$ is said to
be collinear if there exists a line $l$ containing them all.)

\vspace*{0.3cm}

An affine plane is denoted by $(\mathbb{A}, L)$, where $\mathbb{A}$ is the set of all points and $L$ is the set of all lines.

A finite affine plane $(\mathbb{A}, L)$ (that is to say, with a finite number of points) is of order $n$, if there exists a line in $L$ containing exactly $n$ points.
\end{defn}

We say that two lines of an affine plane $(\mathbb{A}, L)$ are {\sl parallel} if they are equal, or have no point in common. To each set of mutually parallel lines we add a single new point incident with each line of this set. The point added is distinct for each such set. These new points are called {\sl points at infinity}. We also add a new line incident with all the points at infinity (and no other points). This line is called {\sl the line at infinity}.

The completion (or closure) $(\Pi, \mathbb{L})$ of an affine plane $(\mathbb{A}, L)$ is defined as follows. The points of $(\Pi, \mathbb{L})$ are the points
of $\mathbb{A}$, plus all the points at infinity of $(\mathbb{A}, L)$. A line in $(\Pi, \mathbb{L})$ is either

(a) an ordinary line $l$ of $L$, plus the point at infinity of $l$, or

(b) the line at infinity, consisting of all the  points at infinity of $(\mathbb{A}, L)$.

The completion $(\Pi, \mathbb{L})$ is a projective plane (see {\cite{hart}}).  
A finite projective plane $(\Pi, \mathbb{L})$ (that is to say, with a finite number of points) is of order $n$, if there exists a line in $\mathbb{L}$ containing exactly $n+1$ points.

\section{Affine and projective planes associated with ternionic projective lines}

An important role in our consideration will be played by pairs not contained in any point of  the projective line, so called, {\sl outliers} (as first shown in \cite{shpp}, see also \cite{hs,sp}). To be more precise, we are interested in those outliers which generate free cyclic submodules.

\begin{theorem}\label{nfcs}
There are exactly $q+1$ free cyclic submodules of \ ${^2T(q)}$ generated by outliers. These are: 
$$T(q)\left(\left[\begin{array}{cclr}
1&0\\ 0&0\end{array}\right],\left[\begin{array}{cclr}
0&0\\1&0\end{array}\right]\right), T(q)\left(\left[\begin{array}{cclr}
k&0\\ 1&0\end{array}\right],\left[\begin{array}{cclr}
1&0\\0&0\end{array}\right]\right),$$
where $k$ runs through all the elements of $F(q).$
\end{theorem}
\begin{proof}
In the light of \cite[Theorem 1.2]{fcs} all non-unimodular free cyclic submodules of ${^2T(q)}$ are generatd by outliers. 

A non-unimodular pair $\left(\left[\begin{array}{cclr}
x_{11}&0\\ x_{21}&x_{22}\end{array}\right],\left[\begin{array}{cclr}
y_{11}&0\\y_{21}&y_{22}\end{array}\right]\right)\subset {^2T(q)}$  generates free cyclic submodule of ${^2T(q)}$ if, and only if, $x_{22}=y_{22}=0$ and one of the following conditions is satisfied:
\begin{enumerate}
\item $x_{11}=0, y_{11}\neq 0, x_{21}\neq 0$;
\item $x_{11}\neq 0, y_{11}=0, y_{21}\neq 0$;
\item $x_{11}\neq 0, y_{11}\neq 0, x_{21}=0, y_{21}\neq 0$;
\item $x_{11}\neq 0, y_{11}\neq 0, x_{21}\neq 0, y_{21}\neq x_{11}^{-1}x_{21}y_{11}$.
\end{enumerate}
Hence the total number of outliers generating free cyclic submodules of ${^2T(q)}$ is $(q-1)^2(q+1)q$ and, according to Remark \ref{generators.of.fcs}, any such submodule is generated by $(q-1)^2q$ distinct outliers. Thus the number of non-unimodular free cyclic submodules of ${^2T(q)}$ is $q+1$. 

Simple calculations show that the listed non-unimodular free cyclic submodules are  pairwise distinct, what completes the proof.

\end{proof}

\begin{rem}\label{nonfree}
A non-free cyclic submodule of ${^2T(q)}$ is one of the following forms:
\begin{enumerate}

\item\label{rank1} of order $1$: 
$T(q)\left(\left[\begin{array}{cclr}
0&0\\ 0&0\end{array}\right],\left[\begin{array}{cclr}
0&0\\0&0\end{array}\right]\right);$

\item\label{rankq} of order $q$:
$T(q)\left(\left[\begin{array}{cclr}
0&0\\ p_{21}&p_{22}\end{array}\right],\left[\begin{array}{cclr}
0&0\\ r_{21}&r_{22}\end{array}\right]\right)$, where $p_{21}, p_{22}, r_{21}, r_{22}\in F(q)$ and not all of them are zero;

\item\label{rankq2} of order $q^2$:
\begin{itemize}
\item $T(q)\left(\left[\begin{array}{cclr}
0&0\\ 0&0\end{array}\right],\left[\begin{array}{cclr}
r_{11}&0\\ r_{21}&0\end{array}\right]\right),$  where $r_{11}\in F^*(q), r_{21}\in F(q)$;
\item 
$T(q)\left(\left[\begin{array}{cclr}
p_{11}&0\\ 0&0\end{array}\right],\left[\begin{array}{cclr}
r_{11}&0\\ 0&0\end{array}\right]\right),$ where $p_{11}\in F^*(q), r_{11}\in F(q)$;
\item
$T(q)\left(\left[\begin{array}{cclr}
p_{11}&0\\ p_{21}&0\end{array}\right],\left[\begin{array}{cclr}
0&0\\ 0&0\end{array}\right]\right)$, where $p_{11}, p_{21}\in F^*(q)$;
\item
$T(q)\left(\left[\begin{array}{cclr}
p_{11}&0\\ p_{21}&0\end{array}\right],\left[\begin{array}{cclr}
p_{11}r_{21}p_{21}^{-1}&0\\ r_{21}&0\end{array}\right]\right)$, where $p_{11}, p_{21}, r_{21}\in F^*(q)$.\end{itemize}
\end{enumerate}
\end{rem}

\bigskip
Our focus will be on those of order $q$. 
\begin{prop}
There are two types of cyclic submodules of ${^2T(q)}$ of order $q$, namely
\begin{enumerate}[(a)]
\item\label{projective} with both entries from $J$; these are 
$T(q)\left(\left[\begin{array}{cclr}
0&0\\ 1&0\end{array}\right],\left[\begin{array}{cclr}
0&0\\0&0\end{array}\right]\right)$, 
and\linebreak
$T(q)\left(\left[\begin{array}{cclr}
0&0\\ k &0\end{array}\right],\left[\begin{array}{cclr}
0&0\\1&0\end{array}\right]\right)$, 
with $k$ running through all the elements of $F(q)$;

\item\label{affine} with at least one entry not belonging to $J$; the latter form $q+1$ pairwise disjoint sets:
\begin{itemize}
\item a single set (hencefort referred to as the first set)
$$\left\{T(q)\left(\left[\begin{array}{cclr}
0&0\\ p_{21}&1\end{array}\right],\left[\begin{array}{cclr}
0&0\\r_{21}&0\end{array}\right]\right), p_{21}, r_{21}\in F(q)\right\} $$

\noindent
and

\item $q$ disjoint sets ($k$-sets) 
$$\left\{T(q)\left(\left[\begin{array}{cclr}
0&0\\ p_{21}&k\end{array}\right],\left[\begin{array}{cclr}
0&0\\r_{21}&1\end{array}\right]\right), p_{21}, r_{21}\in F(q)\right\},$$

\noindent
where, as above, $k$ runs through all the elements of $F(q)$.\end{itemize}
\end{enumerate}
\end{prop}
\begin{proof} 
Let not all $p_{21}, p_{22}, r_{21}, r_{22}\in F(q)$ be equal zero.  By Remark \ref{generators.of.fcs}, two pairs 
$\left(\left[\begin{array}{cclr}
0&0\\ p_{21}&p_{22}\end{array}\right],\left[\begin{array}{cclr}
0&0\\ r_{21}&r_{22}\end{array}\right]\right), \left(\left[\begin{array}{cclr}
0&0\\ s_{21}&s_{22}\end{array}\right],\left[\begin{array}{cclr}
0&0\\ t_{21}&t_{22}\end{array}\right]\right)$ generate the same cyclic submodule if, and only if, there exists $u\in F^*(q)$ such that $s_{21}=up_{21},  t_{21}=ur_{21}, s_{22}=up_{22}, t_{22}=ur_{22}$. 
Thus a cyclic submodule $T(q)\left(\left[\begin{array}{cclr}
0&0\\ p_{21}&p_{22}\end{array}\right],\left[\begin{array}{cclr}
0&0\\ r_{21}&r_{22}\end{array}\right]\right)$ of order $q$ is generated by $q-1$ distinct pairs.

The number of all pairs $\left(\left[\begin{array}{cclr}
0&0\\ p_{21}&0\end{array}\right],\left[\begin{array}{cclr}
0&0\\ r_{21}&0\end{array}\right]\right)$, where $p_{21}, r_{21}\in F(q)$ and $p_{21}\neq 0 \vee r_{21}\neq 0$ is $q^2-1$, so the number of all cyclic submodules generated by them is $q+1$. 

The number of all pairs $\left(\left[\begin{array}{cclr}
0&0\\ p_{21}&p_{22}\end{array}\right],\left[\begin{array}{cclr}
0&0\\ r_{21}&r_{22}\end{array}\right]\right)$, where $p_{21}, p_{22}, r_{21},\linebreak
 r_{22}\in F(q)$ and $p_{22}\neq 0 \vee r_{22}\neq 0$ is $q^2(q^2-1)$, thereby the number of all cyclic submodules generated by them is $q^2(q+1)$.  

It is easy to check that submodules listed in part (\ref{projective}) and submodules in all sets of part (\ref{affine}) are pairwise  distinct which leads to the desired claim.
\end{proof}

A crucial property of submodules of type (\ref{affine}) is stated in the following result:
\begin{prop}\label{crucial}
Submodules of type (\ref{affine}) are the only non-free cyclic submodules of ${^2T(q)}$ not contained in any free cyclic submodule of ${^2T(q)}$ generated by an outlier.
\end{prop}
\begin{proof}
Let a submodule $T(q)\left(\left[\begin{array}{cclr}
x_{11}&0\\ x_{21}&0\end{array}\right],\left[\begin{array}{cclr}
y_{11}&0\\y_{21}&0\end{array}\right]\right)\subset {^2T(q)}$ be free, then the pair 
$\left(\left[\begin{array}{cclr}
a_{11}x_{11}&0\\ a_{21}x_{11}+a_{22}x_{21}&0\end{array}\right],\left[\begin{array}{cclr}
a_{11}y_{11}&0\\ a_{21}y_{11}+a_{22}y_{21}&0\end{array}\right]\right)$ is equal to the pair $\left(\left[\begin{array}{cclr}
0&0\\ p_{21}&p_{22}\end{array}\right],\left[\begin{array}{cclr}
0&0\\r_{21}&r_{22}\end{array}\right]\right)$ if, and only if, $a_{11}=0, p_{22}=r_{22}=0,$ $p_{21}=a_{21}x_{11}+a_{22}x_{21}, r_{21}=a_{21}y_{11}+a_{22}y_{21}$,
thus any submodule of type (\ref{affine}) is not contained in any non-unimodular free cyclic submodule, and any submodule of type (\ref{projective}) is contained in all non-unimodular free cyclic submodules. 

Of course, $T(q)\left(\left[\begin{array}{cclr}
0&0\\ 0&0\end{array}\right],\left[\begin{array}{cclr}
0&0\\0&0\end{array}\right]\right)\subset {^2T(q)}$ is contained in all cyclic submodules of ${^2T(q)}$. 
From Remark \ref{nonfree}(\ref{rankq2}) it follows that the number of all pairs  generating cyclic submodules of ${^2T(q)}$ of order $q^2$ is $q(q^2-1)$ and any such submodule is generated by $q(q-1)$ distinct pairs. Hence the number of all distinct cyclic submodules of ${^2T(q)}$ of order $q^2$ is $(q+1)$. 
By simple calculations we get that 
$T(q)\left(\left[\begin{array}{cclr}
1&0\\ 0&0\end{array}\right],\left[\begin{array}{cclr}
0&0\\0&0\end{array}\right]\right), T(q)\left(\left[\begin{array}{cclr}
k&0\\ 0&0\end{array}\right],\left[\begin{array}{cclr}
1&0\\0&0\end{array}\right]\right),$ 
where $k$ runs through all the elements of $F(q)$, are all distinct cyclic submodules  of ${^2T(q)}$ of order $q^2$. 
For any $k\in F(q)$ they are contained in non-unimodular free cyclic submodules 
$T(q)\left(\left[\begin{array}{cclr}
1&0\\ 0&0\end{array}\right],\left[\begin{array}{cclr}
0&0\\1&0\end{array}\right]\right)$,
$T(q)\left(\left[\begin{array}{cclr}
k&0\\ 1&0\end{array}\right],\left[\begin{array}{cclr}
1&0\\0&0\end{array}\right]\right)$ of ${^2T(q)}$, respectively. It can be shown that these are the only possibilities for submodules of order  $q^2$ to lie in non-unimodular free cyclic submodules of ${^2T(q)}$.
According to Remark \ref{nonfree} there is no other  non-free cyclic submodule of ${^2T(q)}$, which completes the proof.
\end{proof}

\begin{theorem}\label{points}
There are exactly $q(q+1)^2$ points of the projective line $\mathbb{P}(T(q))$. They can be presented as $q+1$ following  sets:

\vspace*{0.3cm}

the first set 

\vspace*{0.3cm}

$\left\{T(q)\left(I,\left[\begin{array}{cclr}
y_{11}&0\\y_{21}&0\end{array}\right]\right), T(q)\left(\left[\begin{array}{cclr}
0&0\\ x_{21}&1\end{array}\right],\left[\begin{array}{cclr}
1&0\\0&0\end{array}\right]\right); x_{21}, y_{11}, y_{21}\in F(q)\right\}$;

\vspace*{0.3cm}

and $q$ $k$-sets

\vspace*{0.3cm}

$\left\{T(q)\left(\left[\begin{array}{cclr}
x_{11}&0\\ x_{21}&k\end{array}\right],I\right), T(q)\left(\left[\begin{array}{cclr}
1&0\\ 0&k\end{array}\right],\left[\begin{array}{cclr}
0&0\\y_{21}&1\end{array}\right]\right); x_{11}, x_{21}, y_{21}\in F(q)\right\}$, 

\vspace*{0.3cm}

where $k$ runs through all the elements of $F(q)$ and $I$ denotes the identity matrix.

Any submodule of the first set of type (\ref{affine}) is contained exclusively in some free cyclic submodules of the first set, and any submodule of the $k$-set of the type (\ref{affine}) is contained exclusively in some free cyclic submodules of the $k$-set.
\end{theorem}

\begin{proof} 
According to \cite[Corollary 1]{graph} there are $q(q+1)^2$ unimodular free cyclic submodules.
It is easy to show that the above-listed sets of points of  $\mathbb{P}(T(q))$ are pairwise disjoint and each of them contains $q^2+q$ distinct free cyclic submodules.

Let $\left(\left[\begin{array}{cclr}
x_{11}&0\\ x_{21}&x_{22}\end{array}\right],\left[\begin{array}{cclr}
y_{11}&0\\ y_{21}&y_{22}\end{array}\right]\right)\subset{^2T(q)}$ be a unimodular pair. Then the pair
$\left(\left[\begin{array}{cclr}
a_{11}x_{11}&0\\ a_{21}x_{11}+a_{22}x_{21}&a_{22}x_{22}\end{array}\right],\left[\begin{array}{cclr}
a_{11}y_{11}&0\\a_{21}y_{11}+a_{22}y_{21}&a_{22}y_{22}\end{array}\right]\right)$ is equal to the pair $\left(\left[\begin{array}{cclr}
0&0\\ p_{21}&1\end{array}\right],\left[\begin{array}{cclr}
0&0\\r_{21}&0\end{array}\right]\right)$ if, and only if, $a_{11}=0, x_{22}\in F^*(q),$\linebreak $y_{22}=0, a_{22}=x_{22}^{-1}, p_{21}=a_{21}x_{11}+x_{22}^{-1}x_{21}, r_{21}=a_{21}y_{11}+x_{22}^{-1}y_{21}$, 
and similarly the pair 
$\left(\left[\begin{array}{cclr}
a_{11}x_{11}&0\\ a_{21}x_{11}+a_{22}x_{21}&a_{22}x_{22}\end{array}\right],\left[\begin{array}{cclr}
a_{11}y_{11}&0\\a_{21}y_{11}+a_{22}y_{21}&a_{22}y_{22}\end{array}\right]\right)$ is equal to the pair $\left(\left[\begin{array}{cclr}
0&0\\ p_{21}&k\end{array}\right],\left[\begin{array}{cclr}
0&0\\r_{21}&1\end{array}\right]\right)$ if, and only if, $a_{11}=0, y_{22}\in F^*(q),$\linebreak $a_{22}=y_{22}^{-1}, k=y_{22}^{-1}x_{22}, p_{21}=a_{21}x_{11}+y_{22}^{-1}x_{21}, r_{21}=a_{21}y_{11}+y_{22}^{-1}y_{21}$.

Therefore any submodule of type $(\ref{affine})$ from a given set is contained in the corresponding set of points of the projective line $\mathbb{P}(T(q))$.
\end{proof}

\begin{cor}\label{distinct.set.points}
Two unimodular free cyclic submodules  of the form\break $T(q)\left(\left[\begin{array}{cclr}
x_{11}&0\\ x_{21}&x_{22}\end{array}\right],\left[\begin{array}{cclr}
y_{11}&0\\ y_{21}&y_{22}\end{array}\right]\right)$, $T(q)\left(\left[\begin{array}{cclr}
w_{11}&0\\ w_{21}&w_{22}\end{array}\right],\left[\begin{array}{cclr}
z_{11}&0\\ z_{21}&z_{22}\end{array}\right]\right)$ are\linebreak in the same set of points of  $\mathbb{P}(T(q))$ if, and only if, $x_{22}=w_{22}, y_{22}=z_{22}$. 
\end{cor}

\begin{ex}\label{first.set}
Let $T(q)\left(\left[\begin{array}{cclr}
1&0\\ 0&1\end{array}\right],\left[\begin{array}{cclr}
y_{11}&0\\y_{21}&0\end{array}\right]\right), T(q)\left(\left[\begin{array}{cclr}
0&0\\ x_{21}&1\end{array}\right],\left[\begin{array}{cclr}
1&0\\0&0\end{array}\right]\right)$ be points of $\mathbb{P}(T(q))$ and let $T(q)\left(\left[\begin{array}{cclr}
0&0\\ p_{21}&1\end{array}\right],\left[\begin{array}{cclr}
0&0\\r_{21}&0\end{array}\right]\right)$ be submodules of type (\ref{affine}). Then
$$\left(\left[\begin{array}{cclr}
a_{11}&0\\ a_{21}&a_{22}\end{array}\right],\left[\begin{array}{cclr}
a_{11}y_{11}&0\\a_{21}y_{11}+a_{22}y_{21}&0\end{array}\right]\right)=\left(\left[\begin{array}{cclr}
0&0\\ p_{21}&1\end{array}\right],\left[\begin{array}{cclr}
0&0\\r_{21}&0\end{array}\right]\right)$$ if, and only if, $a_{11}=0, a_{22}=1, a_{21}=p_{21}, r_{21}=p_{21}y_{11}+y_{21}$ for any $p_{21}\in F(q)$, and 

$$\left(\left[\begin{array}{cclr}
0&0\\ a_{22}x_{21}&a_{22}\end{array}\right],\left[\begin{array}{cclr}
a_{22}&0\\a_{21}&0\end{array}\right]\right)=\left(\left[\begin{array}{cclr}
0&0\\ p_{21}&1\end{array}\right],\left[\begin{array}{cclr}
0&0\\r_{21}&0\end{array}\right]\right)$$ if, and only if, $a_{11}=0, a_{22}=1, a_{21}=r_{21}, p_{21}=x_{21}$ for any $r_{21}\in F(q)$.

Hence
$$T(q)\left(\left[\begin{array}{cclr}
0&0\\ p_{21}&1\end{array}\right],\left[\begin{array}{cclr}
0&0\\p_{21}y_{11}+y_{21}&0\end{array}\right]\right)\subset T(q)\left(\left[\begin{array}{cclr}
1&0\\ 0&1\end{array}\right],\left[\begin{array}{cclr}
y_{11}&0\\y_{21}&0\end{array}\right]\right)$$ for any $p_{21}, y_{11}, y_{21}\in F(q)$, and

$$T(q)\left(\left[\begin{array}{cclr}
0&0\\ x_{21}&1\end{array}\right],\left[\begin{array}{cclr}
0&0\\r_{21}&0\end{array}\right]\right)\subset T(q)\left(\left[\begin{array}{cclr}
0&0\\ x_{21}&1\end{array}\right],\left[\begin{array}{cclr}
1&0\\0&0\end{array}\right]\right)$$ for any $x_{21}, r_{21}\in F(q)$, and by the proof of Proposition \ref{points} these are the only submodules of type (\ref{affine})  contained in the considered points of $\mathbb{P}(T(q))$.

\end{ex}

\begin{theorem}\label{affine.plane}
Let us regard cyclic submodules contained in the set $$\left\{T(q)\left(\left[\begin{array}{cclr}
0&0\\ p_{21}&p_{22}\end{array}\right]\left[\begin{array}{cclr}
0&0\\ r_{21}&r_{22}\end{array}\right]\right); p_{21}, r_{21}\in F(q)\right\},$$ where  $p_{22}, r_{22}$ are fixed elements of $F(q)$ such that $p_{22}\neq 0 \vee  r_{22}\neq 0$, as points and free cyclic submodules containing them as lines. 

Then these points and lines form a point-line incidence structure isomorphic to the {\it affine} plane of order $q$.
\end{theorem}
\begin{proof}
Consider the set $\left\{T(q)\left(\left[\begin{array}{cclr}
0&0\\ p_{21}&1\end{array}\right]\left[\begin{array}{cclr}
0&0\\ r_{21}&0\end{array}\right]\right); p_{21}, r_{21}\in F(q)\right\}$.

Suppose that $T(q)\left(\left[\begin{array}{cclr}
0&0\\ p_{21}&1\end{array}\right]\left[\begin{array}{cclr}
0&0\\ r_{21}&0\end{array}\right]\right), T(q)\left(\left[\begin{array}{cclr}
0&0\\ s_{21}&1\end{array}\right]\left[\begin{array}{cclr}
0&0\\ t_{21}&0\end{array}\right]\right)$\linebreak are contained in the point $ T(q)\left(\left[\begin{array}{cclr}
1&0\\ 0&1\end{array}\right],\left[\begin{array}{cclr}
y_{11}&0\\y_{21}&0\end{array}\right]\right)\in  \mathbb{P}(T(q))$.
By Example \ref{first.set} this is equivalent to saying that  $p_{21}y_{11}+y_{21}=r_{21}$ and $s_{21}y_{11}+y_{21}=t_{21}$. Consequently $(s_{21}-p_{21})y_{11}=t_{21}-r_{21}$.

If $s_{21}=p_{21}$, then $t_{21}=r_{21}$, so the considered submodules are not distinct.
If $s_{21}\neq p_{21}$, then $y_{11}=(s_{21}-p_{21})^{-1}(t_{21}-r_{21})$. Thus  there exists exactly one $y_{11}$ and exactly one $y_{21}$ for fixed $p_{21}, r_{21}, s_{21}, t_{21}\in F(q)$.

Let now that
$T(q)\left(\left[\begin{array}{cclr}
0&0\\ p_{21}&1\end{array}\right]\left[\begin{array}{cclr}
0&0\\ r_{21}&0\end{array}\right]\right), T(q)\left(\left[\begin{array}{cclr}
0&0\\ s_{21}&1\end{array}\right]\left[\begin{array}{cclr}
0&0\\ t_{21}&0\end{array}\right]\right)$ be contained in the point $ T(q)\left(\left[\begin{array}{cclr}
0&0\\ x_{21}&1\end{array}\right],\left[\begin{array}{cclr}
1&0\\0&0\end{array}\right]\right)\in  \mathbb{P}(T(q))$.

By Example \ref{first.set} this is equivalent to saying that  $p_{21}=s_{21}=x_{21}$. 

This shows that two distinct submodules of the considered set are contained in exactly one point of  $\mathbb{P}(T(q))$.

By using results of Example \ref{first.set} again we get that
any submodule of the considered set not contained in the point $T(q)\left(\left[\begin{array}{cclr}
1&0\\ 0&1\end{array}\right],\left[\begin{array}{cclr}
y_{11}&0\\y_{21}&0\end{array}\right]\right)\in  \mathbb{P}(T(q))$ is of the form $T(q)\left(\left[\begin{array}{cclr}
0&0\\ p_{21}&1\end{array}\right],\left[\begin{array}{cclr}
0&0\\p_{21}y_{11}+x'_{21}&0\end{array}\right]\right)$, where $x'_{21}, p_{21}\in F(q)$ and $ x'_{21}\neq y_{21}$. 

We get also that
elements of the first set of points of  $\mathbb{P}(T(q))$ which do not contain the submodule $T(q)\left(\left[\begin{array}{cclr}
0&0\\ p_{21}&1\end{array}\right],\left[\begin{array}{cclr}
0&0\\p_{21}y_{11}+y_{21}&0\end{array}\right]\right)$ for any $p_{21}\in F(q)$ are of the form $T(q)\left(\left[\begin{array}{cclr}
1&0\\ 0&1\end{array}\right],\left[\begin{array}{cclr}
y_{11}&0\\y'_{21}&0\end{array}\right]\right)$, where $y'_{21}\in F(q)$ and $y'_{21}\neq y_{21}$. 

Of course, exactly one of them, i.e., $T(q)\left(\left[\begin{array}{cclr}
1&0\\ 0&1\end{array}\right],\left[\begin{array}{cclr}
y_{11}&0\\x'_{21}&0\end{array}\right]\right)\in  \mathbb{P}(T(q))$\linebreak contains $T(q)\left(\left[\begin{array}{cclr}
0&0\\ p_{21}&1\end{array}\right],\left[\begin{array}{cclr}
0&0\\p_{21}y_{11}+x'_{21}&0\end{array}\right]\right)\subset  {^2T(q)}$.

Let $x_{21},y_{21} \in F(q)$ and $y_{21}\neq x_{21}$. In the same manner we get that there exists a unique point  $T(q)\left(\left[\begin{array}{cclr}
0&0\\ y_{21}&1\end{array}\right],\left[\begin{array}{cclr}
1&0\\0&0\end{array}\right]\right)\in  \mathbb{P}(T(q))$, which contains the submodule $T(q)\left(\left[\begin{array}{cclr}
0&0\\ y_{21}&1\end{array}\right],\left[\begin{array}{cclr}
0&0\\r_{21}&0\end{array}\right]\right)\subset {^2T(q)}$ and does not contain the submodule $T(q)\left(\left[\begin{array}{cclr}
0&0\\ x_{21}&1\end{array}\right],\left[\begin{array}{cclr}
0&0\\r_{21}&0\end{array}\right]\right)$ for any $r_{21}\in F(q)$. 

To sum up, given any point $T(q)\left(\left[\begin{array}{cclr}
x_{11}&0\\ x_{21}&1\end{array}\right],\left[\begin{array}{cclr}
y_{11}&0\\ y_{21}&0\end{array}\right]\right)\in\mathbb{P}(T(q))$ of the first set and any submodule $T(q)\left(\left[\begin{array}{cclr}
0&0\\ p_{21}&1\end{array}\right]\left[\begin{array}{cclr}
0&0\\ r_{21}&0\end{array}\right]\right)\subset {^2T(q)}$, which is not contained in that point of  $\mathbb{P}(T(q))$, there exists a unique point\break $T(q)\left(\left[\begin{array}{cclr}
x'_{11}&0\\ x'_{21}&1\end{array}\right],\left[\begin{array}{cclr}
y'_{11}&0\\ y'_{21}&0\end{array}\right]\right)\in\mathbb{P}(T(q))$, which contains that submodule and does not contain any submodule of type $(\ref{affine})$, which is contained in $T(q)\left(\left[\begin{array}{cclr}
x_{11}&0\\ x_{21}&1\end{array}\right],\left[\begin{array}{cclr}
y_{11}&0\\ y_{21}&0\end{array}\right]\right)\in\mathbb{P}(T(q))$.

Suppose that the following three submodules: $T(q)\left(\left[\begin{array}{cclr}
0&0\\ 0&1\end{array}\right],\left[\begin{array}{cclr}
0&0\\0&0\end{array}\right]\right),$\break $ T(q)\left(\left[\begin{array}{cclr}
0&0\\ 0&1\end{array}\right],\left[\begin{array}{cclr}
0&0\\1&0\end{array}\right]\right),T(q)\left(\left[\begin{array}{cclr}
0&0\\ 1&1\end{array}\right],\left[\begin{array}{cclr}
0&0\\0&0\end{array}\right]\right)$ are contained in a free cyclic submodule $T(q)\left(\left[\begin{array}{cclr}
x_{11}&0\\ x_{21}&1\end{array}\right],\left[\begin{array}{cclr}
y_{11}&0\\ y_{21}&0\end{array}\right]\right)\in  \mathbb{P}(T(q))$. Then there exist $a_{21}, b_{21}, c_{21}\in F(q)$ such that
$$\begin{cases}
1.\ a_{21}x_{11}+x_{21}=0.\\2.\ a_{21}y_{11}+y_{21}=0.\\3.\ b_{21}x_{11}+x_{21}=0.\\4.\ b_{21}y_{11}+y_{21}=1.\\5.\ c_{21}x_{11}+x_{21}=1.\\6.\ c_{21}y_{11}+y_{21}=0.\end{cases}$$ Thus $(b_{21}-a_{21})x_{11}=0, (b_{21}-a_{21})y_{11}=1$. Hence $b_{21}\neq a_{21}, x_{11}=0$ and $x_{21}=-a_{21}x_{11}=0$. But then we get $c_{21}x_{11}+x_{21}=0$, which contradicts   equation (5) and thereby it contradicts the assumption that the three above  submodules are contained in the same point of  $\mathbb{P}(T(q))$.  
It means that there exist three submodules $T(q)\left(\left[\begin{array}{cclr}
0&0\\ p_{21}&1\end{array}\right]\left[\begin{array}{cclr}
0&0\\ r_{21}&0\end{array}\right]\right)\subset {^2T(q)}$ not contained in the same point of  $\mathbb{P}(T(q))$.

Moreover, any point $\mathbb{P}(T(q))$ of the first set contains exactly $q$ submodules of the form $T(q)\left(\left[\begin{array}{cclr}
0&0\\ p_{21}&1\end{array}\right]\left[\begin{array}{cclr}
0&0\\ r_{21}&0\end{array}\right]\right)\subset {^2T(q)}$, what follows directly from Example \ref{first.set}.  

We have just shown that all axioms of an affine plane are satisfied for the considered set.

By the same methods it follows that submodules of type $(\ref{affine})$ from a given $k$-set and points of  $\mathbb{P}(T(q))$ containing them give a point-line incidence structure isomorphic to the affine plane of order $q$. So, there are altogether $q+1$ isomorphic affine planes of order $q$ associated with the projective line  $\mathbb{P}(T(q))$.
\end{proof} 

\begin{cor}\label{planes.classes.lines}
\begin{enumerate}
\item\label{planes} Two unimodular free cyclic submodules generated by pairs $\left(\left[\begin{array}{cclr}
x_{11}&0\\ x_{21}&x_{22}\end{array}\right],\left[\begin{array}{cclr}
y_{11}&0\\ y_{21}&y_{22}\end{array}\right]\right)$, $\left(\left[\begin{array}{cclr}
w_{11}&0\\ w_{21}&w_{22}\end{array}\right],\left[\begin{array}{cclr}
z_{11}&0\\ z_{21}&z_{22}\end{array}\right]\right)$ represent lines of the same affine plane associated with  $\mathbb{P}(T(q))$ if, and only if, $x_{22}=w_{22}, y_{22}=z_{22}.$
\item\label{classes} Two unimodular free cyclic submodules generated by pairs of the form
$\left(\left[\begin{array}{cclr}
x_{11}&0\\ x_{21}&x_{22}\end{array}\right],\left[\begin{array}{cclr}
y_{11}&0\\ y_{21}&y_{22}\end{array}\right]\right)$, $\left(\left[\begin{array}{cclr}
w_{11}&0\\ w_{21}&x_{22}\end{array}\right],\left[\begin{array}{cclr}
z_{11}&0\\ z_{21}&y_{22}\end{array}\right]\right)$ represent the same class of parallel lines of an affine plane associated with  $\mathbb{P}(T(q))$ if, and only if, $x_{11}=w_{11}, y_{11}=z_{11}.$
\item\label{lines} Two unimodular free cyclic submodules  generated by pairs of the form $\left(\left[\begin{array}{cclr}
x_{11}&0\\ x_{21}&x_{22}\end{array}\right],\left[\begin{array}{cclr}
y_{11}&0\\ y_{21}&y_{22}\end{array}\right]\right)$, $\left(\left[\begin{array}{cclr}
x_{11}&0\\ w_{21}&x_{22}\end{array}\right],\left[\begin{array}{cclr}
y_{11}&0\\ z_{21}&y_{22}\end{array}\right]\right)$ represent distinct parallel lines of an affine plane associated with  $\mathbb{P}(T(q))$ if, and only if, $x_{21}\neq w_{21}$ or $y_{21}\neq z_{21}.$
\end{enumerate}
\end{cor}
This is, however, not a full story as it turns out that any of these affine plane can be extended to the {\it projective} plane of order $q$ due to the fact that in  $^2T(q)$ there also exist non-unimodular pairs (outliers) generating free cyclic submodules.
\begin{theorem}\label{projective.plane}
Any affine plane of order $q$ associated with  $\mathbb{P}(T(q))$ can be extended to the projective plane of order $q$ in the following way.

Consider the set of unimodular free cyclic submodules generated by pairs  $\left(\left[\begin{array}{cclr}
x_{11}&0\\ x_{21}&x_{22}\end{array}\right],\left[\begin{array}{cclr}
y_{11}&0\\ y_{21}&y_{22}\end{array}\right]\right)\subset  {^2T(q)}$, such that $x_{22}, y_{22}\in F(q)$ are fixed. 
\begin{enumerate}
\item For all such submodules, where $x_{11}, y_{11}\in F(q)$ are also fixed,  a submodule  $T(q)\left(\left[\begin{array}{cclr}
0&0\\ x_{11}&0\end{array}\right],\left[\begin{array}{cclr}
0&0\\y_{11}&0\end{array}\right]\right)\subset {^2T(q)}$ must be taken into account as a new point, and
\item A  free cyclic submodule  $T(q)\left(\left[\begin{array}{cclr}
x_{22}&0\\ y_{22}&0\end{array}\right],\left[\begin{array}{cclr}
y_{22}&0\\\delta_{y_{22}0}&0\end{array}\right]\right)\subset {^2T(q)}$,\linebreak where $\delta_{y_{22}0}$ stands for the Kronecker delta, must be taken into account as a new line. 
\end{enumerate} 
Points of the obtained projective plane are all cyclic submodules of order $q$ of ${^2T(q)}$, which are contained in a given set of unimodular free cyclic submodules of  $\mathbb{P}(T(q))$, and lines are these free cyclic submodules and the respective free cyclic submodule generated by an outlier. 
\end{theorem}
\begin{proof}
Points of the form $T(q)\left(\left[\begin{array}{cclr}
x_{11}&0\\ x_{21}&x_{22}\end{array}\right],\left[\begin{array}{cclr}
y_{11}&0\\ y_{21}&y_{22}\end{array}\right]\right)\in  \mathbb{P}(T(q))$, such that $x_{22}, y_{22}\in F(q)$ are fixed, represent all lines of an affine plane associated with  $\mathbb{P}(T(q))$. 
\begin{enumerate}
\item 
Obviously, a new submodule $T(q)\left(\left[\begin{array}{cclr}
0&0\\ x_{11}&0\end{array}\right],\left[\begin{array}{cclr}
0&0\\y_{11}&0\end{array}\right]\right)\subset {^2T(q)}$ is contained in each submodule $T(q)\left(\left[\begin{array}{cclr}
x_{11}&0\\ x_{21}&x_{22}\end{array}\right],\left[\begin{array}{cclr}
y_{11}&0\\ y_{21}&y_{22}\end{array}\right]\right)\in \mathbb{P}(T(q))$, which, according to  Corollary \ref{planes.classes.lines} (\ref{classes}),  represents a line of a given parallel class, and is distinct for each such class of an affine plane associated with  $\mathbb{P}(T(q))$. 

\item A cyclic submodule $T(q)\left(\left[\begin{array}{cclr}
x_{22}&0\\ y_{22}&0\end{array}\right],\left[\begin{array}{cclr}
y_{22}&0\\\delta_{y_{22}0}&0\end{array}\right]\right)\subset {^2T(q)}$ 
contains a submodule $T(q)\left(\left[\begin{array}{cclr}
0&0\\ x_{11}&0\end{array}\right],\left[\begin{array}{cclr}
0&0\\y_{11}&0\end{array}\right]\right)\subset {^2T(q)}$ for any $x_{11}, y_{11}\in F(q)$,
and it
does not contain any submodule of type (\ref{affine}), this is straightforward from Proposition \ref{crucial} and its proof.

It means that a new line (represented by  a non-unimodular free cyclic submodule) is incident with all new points (represented by submodules of type (\ref{projective})) and no other points. 
\end{enumerate}
\end{proof}

Summing up, cyclic submodules of order $q$ generated by pairs with both entries in $J$ are the points and  any non-unimodular free cyclic submodule incorporating these submodules is the line representing the {\it unique} projective closure of all affine planes associated with $\mathbb{P}(T(q))$.
Figure 1 serves as an illustration of our findings for the case $q=2$.

\begin{figure}[pth!]
\centerline{\includegraphics[width=11truecm,clip=]{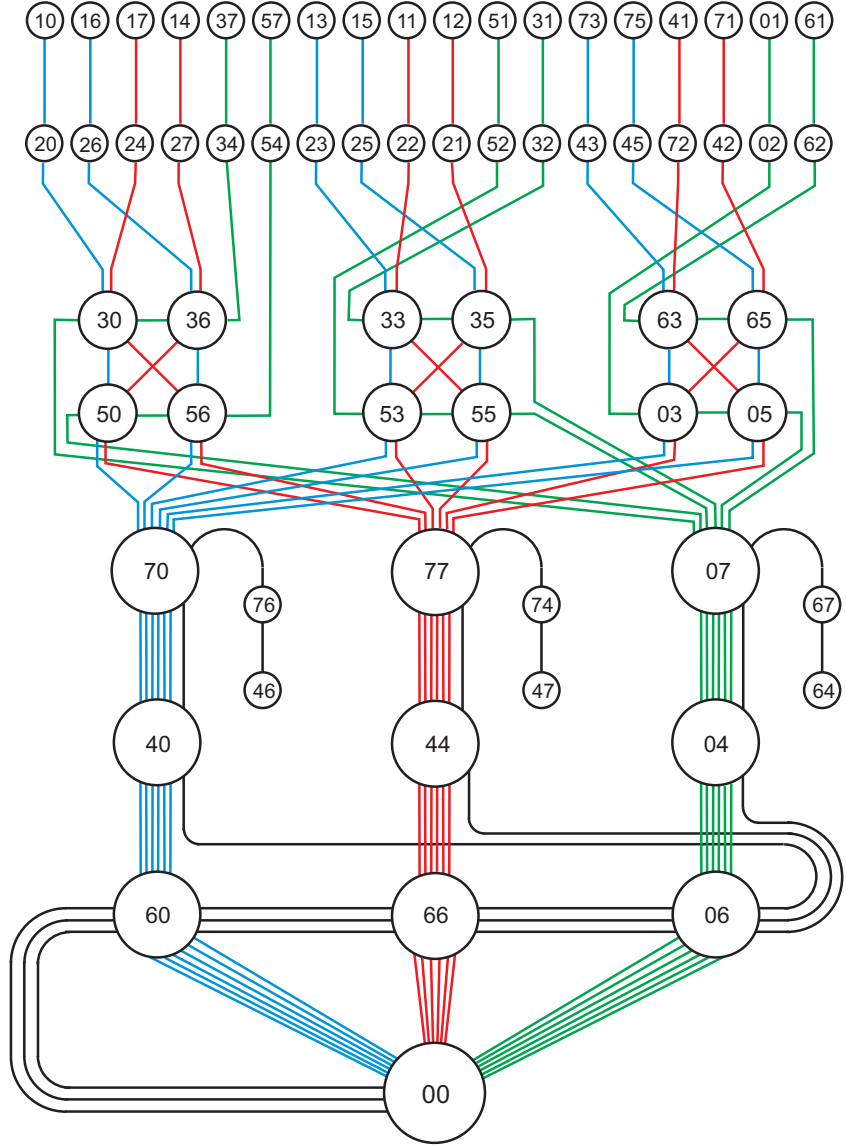}}
\caption{A visualisation of all free cyclic submodules over $T(2)$ after \cite{sp}. The 18 free cyclic submodules generated by unimodular pairs are represented by colored broken polygons, whereas those three generated by nonunimodular vectors are represented by black curves. For the sake of simplicity, the matrices of $T(2)$ are labeled by integers from 0 to 7 in such a way that $J = \{0,6\}$ (for more details, see \cite{shpp,sp}). The three affine planes are represented by the following set of points: 
$\{(3,0),(3,6),(5,0),(5,6)\}$, $\{(3,3),(3,5),(5,3),(5,5)\}$ and $\{(6,3),(6,5),(0,3),(0,5)\}$ and two parallel lines have the same color; the closure line comprises the points $(6,0)$, $(6,6)$ and $(0,6)$. } 
\end{figure}
\bigskip

\section{Affine and projective planes associated with $\mathbb{P}(T_3(q))$}

\begin{theorem}\label{fcs.out.3}
There are exactly $(q+1)^2(2q^2+q+1)$ free cyclic submodules of \ ${^2T_3(q)}$ generated by outliers. They can be presented as  $q+1$ sets: the first set and $q$ $k$-sets, where $k\in F(q)$.

The first set consists of submodules generated by pairs: 
$$\left(\left[\begin{array}{cclr}
1&0&0\\0&0&0\\0&x_{32}&1\end{array}\right],\left[\begin{array}{cclr}
0&0&0\\ 1&0&0\\0&y_{32}&y_{33}\end{array}\right]\right),
\left(\left[\begin{array}{cclr}
1&0&0\\0&0&0\\0&x_{32}&0\end{array}\right],\left[\begin{array}{cclr}
0&0&0\\ 1&0&0\\0&y_{32}&1\end{array}\right]\right),$$
$$\left(\left[\begin{array}{cclr}
1&0&0\\0&0&0\\0&x_{32}&0\end{array}\right],\left[\begin{array}{cclr}
0&0&0\\ 1&0&0\\0&1&0\end{array}\right]\right),\left(\left[\begin{array}{cclr}
1&0&0\\0&0&0\\0&1&0\end{array}\right],\left[\begin{array}{cclr}
0&0&0\\ 1&0&0\\0&0&0\end{array}\right]\right),$$
$$\left(\left[\begin{array}{cclr}
1&0&0\\0&1&0\\0&0&0\end{array}\right],\left[\begin{array}{cclr}
0&0&0\\ y_{21}&y_{22}&0\\y_{31}&1&0\end{array}\right]\right),
\left(\left[\begin{array}{cclr}
1&0&0\\0&1&0\\0&0&0\end{array}\right],\left[\begin{array}{cclr}
0&0&0\\ y_{21}&y_{22}&0\\1&0&0\end{array}\right]\right),$$
$$\left(\left[\begin{array}{cclr}
1&0&0\\0&0&0\\0&1&0\end{array}\right],\left[\begin{array}{cclr}
0&0&0\\ y_{21}&1&0\\y_{31}&0&0\end{array}\right]\right),
\left(\left[\begin{array}{cclr}
1&0&0\\0&0&0\\0&0&0\end{array}\right],\left[\begin{array}{cclr}
0&0&0\\ y_{21}&1&0\\1&0&0\end{array}\right]\right),$$
where $x_{32}, y_{21}, y_{22}, y_{31}, y_{32}, y_{33}\in F(q)$;

\vspace*{0.3cm}

Each of $k$-sets consists of submodules generated by pairs:  

$$\left(\left[\begin{array}{cclr}
k&0&0\\ 1&0&0\\0&x_{32}&x_{33}
\end{array}\right],
\left[\begin{array}{cclr}
1&0&0\\0&0&0\\0&y_{32}&1\end{array}\right]\right),
\left(\left[\begin{array}{cclr}
k&0&0\\ 1&0&0\\0&x_{32}&1\end{array}\right],
\left[\begin{array}{cclr}1&0&0\\0&0&0\\0&y_{32}&0\end{array}\right]\right),$$
$$\left(\left[\begin{array}{cclr}
k&0&0\\ 1&0&0\\0&1&0\end{array}\right],
\left[\begin{array}{cclr}1&0&0\\0&0&0\\0&y_{32}&0\end{array}\right]\right),
\left(\left[\begin{array}{cclr}
k&0&0\\ 1&0&0\\0&0&0\end{array}\right],
\left[\begin{array}{cclr}1&0&0\\0&0&0\\0&1&0\end{array}\right]\right),$$
$$\left(\left[\begin{array}{cclr}
k&0&0\\ x_{21}&x_{22}&0\\x_{31}&1&0\end{array}\right],
\left[\begin{array}{cclr}1&0&0\\0&1&0\\0&0&0\end{array}\right]\right),
\left(\left[\begin{array}{cclr}
k&0&0\\ x_{21}&x_{22}&0\\1&0&0\end{array}\right],
\left[\begin{array}{cclr}1&0&0\\0&1&0\\0&0&0\end{array}\right]\right),$$
$$\left(\left[\begin{array}{cclr}
k&0&0\\ x_{21}&1&0\\x_{31}&0&0\end{array}\right],
\left[\begin{array}{cclr}1&0&0\\0&0&0\\0&1&0\end{array}\right]\right),
\left(\left[\begin{array}{cclr}
k&0&0\\ x_{21}&1&0\\1&0&0\end{array}\right],
\left[\begin{array}{cclr}1&0&0\\0&0&0\\0&0&0\end{array}\right]\right),$$
where $x_{21}, x_{22}, x_{31}, x_{32}, x_{33}, y_{32}\in F(q)$.
\end{theorem}
\begin{proof}
In the light of \cite[Theorem 1.2]{fcs} all non-unimodular free cyclic submodules of ${^2T_3(q)}$ are generatd by outliers. 

A non-unimodular pair $\left(\left[\begin{array}{cclr}
x_{11}&0&0\\ x_{21}&x_{22}&0\\x_{31}&x_{32}&x_{33}\end{array}\right],\left[\begin{array}{cclr}
y_{11}&0&0\\y_{21}&y_{22}&0\\y_{31}&y_{32}&y_{33}\end{array}\right]\right)\subset {^2T_3(q)}$  generates free cyclic submodule of ${^2T_3(q)}$ if, and only if, $x_{22}=y_{22}=0$ or $x_{22}=y_{22}=0$ and one of the following conditions is satisfied:
\begin{enumerate}
\item $x_{11}=0, y_{11}\neq 0$ and
\begin{enumerate} 
\item $x_{22}=y_{22}=0, x_{21}\neq 0, x_{32}\neq 0 \vee x_{33}\neq 0 \vee y_{32}\neq 0 \vee  y_{33}\neq 0,$ or
\item $x_{22}=0, y_{22}\neq 0, x_{31}\neq x_{21}x_{22}^{-1}x_{32} \vee (x_{31}=x_{21}y_{22}^{-1}y_{32}\wedge x_{32}\neq 0),$ or
\item $x_{22}\neq 0, y_{22}= 0, x_{31}\neq x_{21}x_{22}^{-1}x_{32} \vee (x_{31}=x_{21}x_{22}^{-1}x_{32}\wedge y_{32}\neq 0),$ or
\item $x_{22}\neq 0, y_{22}\neq 0, x_{32}\neq x_{22}y_{22}^{-1}y_{32} \vee (x_{32}=x_{22}y_{22}^{-1}y_{32}\wedge x_{31}\neq x_{21}y_{22}^{-1}y_{32}).$
\end{enumerate}
\item $x_{11}\neq 0, y_{11}=0$ and
\begin{enumerate} 
\item $x_{22}=y_{22}=0, y_{21}\neq 0, x_{32}\neq 0 \vee x_{33}\neq 0 \vee y_{32}\neq 0 \vee  y_{33}\neq 0,$ or
\item $x_{22}=0, y_{22}\neq 0, y_{31}\neq y_{21}y_{22}^{-1}y_{32} \vee (y_{31}=y_{21}y_{22}^{-1}y_{32}\wedge x_{32}\neq 0),$ or
\item $x_{22}\neq 0, y_{22}= 0, y_{31}\neq y_{21}x_{22}^{-1}x_{32} \vee (y_{31}=y_{21}x_{22}^{-1}x_{32}\wedge y_{32}\neq 0),$ or
\item $x_{22}\neq 0, y_{22}\neq 0, x_{32}\neq x_{22}y_{22}^{-1}y_{32} \vee (x_{32}=x_{22}y_{22}^{-1}y_{32}\wedge x_{31}\neq x_{21}y_{22}^{-1}y_{32}),$
\end{enumerate}
\item $x_{11}\neq 0, y_{11}\neq 0$ and
\begin{enumerate} 
\item $x_{22}=y_{22}=0, x_{21}\neq x_{11}y_{11}^{-1}y_{21},$ or
\item $x_{22}=0, y_{22}\neq 0, x_{32}\neq 0  \vee \big(x_{32}=0 \wedge x_{31}\neq x_{11}y_{11}^{-1}y_{31}+(x_{21}-x_{11}y_{11}^{-1}y_{21})y_{22}^{-1}y_{32}\big),$ or
\item $x_{22}\neq 0, y_{22}= 0, y_{32}\neq 0  \vee \big(y_{32}=0 \wedge x_{31}\neq x_{11}y_{11}^{-1}y_{31}+(x_{21}-x_{11}y_{11}^{-1}y_{21})x_{22}^{-1}x_{32}\big),$ or
\item $x_{22}\neq 0, y_{22}\neq 0, x_{32}\neq x_{22}y_{22}^{-1}y_{32} \vee \big(x_{32}=x_{22}y_{22}^{-1}y_{32}\wedge x_{31}\neq x_{11}y_{11}^{-1}y_{31}+(x_{21}- x_{11}y_{11}^{-1}y_{21})y_{22}^{-1}y_{32}\big).$
\end{enumerate}
\end{enumerate}
Hence the total number of outliers generating free cyclic submodules of ${^2T_3(q)}$ is $(q-1)^3q^3(q+1)^2(2q^2+q+1)$.  Remark \ref{generators.of.fcs} implies that any such submodule is generated by $(q-1)^3q^3$ distinct outliers. So the number of non-unimodular free cyclic submodules of ${^2T_3(q)}$ is $(q+1)^2(2q^2+q+1)$. 

By using Remark \ref{generators.of.fcs} again and multiplying outliers by all invertible matrices of $T_3^*(q)$ from the left we get immediately that the listed non-unimodular free cyclic submodules are pairwise distinct. 
\end{proof}

Just as in the case of ternions, we are interested in non-free cyclic submodules not contained in any free cyclic submodule generated by outlier.

\begin{prop}
There are exactly $(q+1)^2q^5$ non-free cyclic submodules of ${^2T_3(q)}$ not contained in any non-unimodular free cyclic submodule of ${^2T_3(q)}$. They can be presented as $q+1$ sets having $q^2+q$ subsets each. 

The first set consists of $q^2$ subsets of the form
$$\left\{T_3(q)\left(\left[\begin{array}{cclr}
0&0&0\\ p_{21}&1&0\\p_{31}&0&1\end{array}\right],\left[\begin{array}{cclr}
0&0&0\\r_{21}&0&0\\r_{31}&r_{32}&r_{33}\end{array}\right]\right); p_{21}, p_{31}, r_{21}, r_{31}\in F(q)\right\},$$ 
where $r_{32}, r_{33}$ run through all the elements of  $F(q)$, and $q$ subsets of the form
$$\left\{T_3(q)\left(\left[\begin{array}{cclr}
0&0&0\\ p_{21}&1&0\\p_{31}&0&0\end{array}\right],\left[\begin{array}{cclr}
0&0&0\\r_{21}&0&0\\r_{31}&r_{32}&1\end{array}\right]\right); p_{21}, p_{31}, r_{21}, r_{31}\in F(q)\right\},$$ 
where $r_{32}$ runs through all the elements of  $F(q)$. 

\vspace*{0.3cm}

Each of $q$ $k$-sets, where $k\in F(q)$, consists of $q^2$ subsets of the form
$$\left\{T_3(q)\left(\left[\begin{array}{cclr}
0&0&0\\ p_{21}&k&0\\p_{31}&p_{32}&p_{33}\end{array}\right],\left[\begin{array}{cclr}
0&0&0\\r_{21}&1&0\\r_{31}&0&1\end{array}\right]\right); p_{21}, p_{31}, r_{21}, r_{31}\in F(q)\right\},$$ 
where $p_{32}, p_{33}$ run through all the elements of  $F(q)$, and $q$ subsets of the form
$$\left\{T_3(q)\left(\left[\begin{array}{cclr}
0&0&0\\ p_{21}&k&0\\p_{31}&p_{32}&1\end{array}\right],\left[\begin{array}{cclr}
0&0&0\\r_{21}&1&0\\r_{31}&0&0\end{array}\right]\right); p_{21}, p_{31}, r_{21}, r_{31}\in F(q)\right\},$$
where $p_{32}$ runs through all the elements of  $F(q)$.
\end{prop}
\begin{proof}
It is easy to show that subsets  specified in the assertion  are pairwise disjoint and any of them contains $q^4$ distinct pairs. 

Suppose now that ${T_3(q)\big(X, Y\big)}\subset {^2T_3(q)}$ is a free cyclic submodule  containing a submodule generated by a pair $\left(\left[\begin{array}{cclr}
0&0&0\\ p_{21}&1&0\\p_{31}&0&1\end{array}\right],\left[\begin{array}{cclr}
0&0&0\\r_{21}&0&0\\r_{31}&r_{32}&r_{33}\end{array}\right]\right)$. Then $x_{11}\neq 0 \vee y_{11}\neq 0$ and there exists $A\in T_3(q)$ such that $$\left[\begin{array}{cclr}
a_{11}x_{11}&0&0\\ a_{21}x_{11}+a_{22}x_{21}&a_{22}x_{22}&0\\
a_{31}x_{11}+a_{32}x_{21}+a_{33}x_{31}&a_{32}x_{22}+a_{33}x_{32}&a_{33}x_{33}\end{array}\right]
=\left[\begin{array}{cclr}
0&0&0\\ p_{21}&1&0\\p_{31}&0&1\end{array}\right]$$
 and
$$\left[\begin{array}{cclr}a_{11}z_{11}&0&0\\ a_{21}y_{11}+a_{22}y_{21}&a_{22}y_{22}&0\\
a_{31}y_{11}+a_{32}y_{21}+a_{33}y_{31}&a_{32}y_{22}+a_{33}y_{32}&a_{33}y_{33}\end{array}\right]
=\left[\begin{array}{cclr}
0&0&0\\r_{21}&0&0\\r_{31}&r_{32}&r_{33}\end{array}\right].$$
Hence $x_{22}\neq 0 $ and $x_{33}\neq 0$, so the pair $\big(X, Y\big)$ is unimodular. By the same methods it follows that all remaining specified submodules are not contained in any non-unimodular free cyclic submodule. 

We will show now that any other non-free cyclic submodules of ${^2T_3(q)}$ is contained in some free cyclic submodule generated by an outlier. Of course, $p_{22}=r_{22 }=0$ or $p_{33}=r_{33}=0$ for all pairs $$(P, R)=\left(\left[\begin{array}{cclr}
p_{11}&0&0\\ p_{21}&p_{22}&0\\p_{31}&p_{32}&p_{33}\end{array}\right],\left[\begin{array}{cclr}
r_{11}&0&0\\r_{21}&r_{22}&0\\r_{31}&r_{32}&r_{33}\end{array}\right]\right)\subset {^2T_3(q)}$$ 
generating these non-free cyclic submodules. Below we consider all the cases of such pairs $(P, R)$, giving also the  matrix $A\in T_3(q)$ and  the outlier $(X, Y)\subset {^2T_3(q)}$ generating a free cyclic submodule such that $A(X, Y)=(P, R)$.

\vspace*{0.3cm}

${\bf 1.}$ $p_{11}\neq 0$; 
${\bf 1.1}$  $p_{22}=r_{22}=0$;

${\bf 1.1.1.}$ $p_{32}\neq 0\vee p_{33}\neq 0 \vee r_{32}\neq 0\vee r_{33}\neq 0$, $r_{21}=p_{21}p_{11}^{-1}r_{11}$;  

\vspace*{0.3cm}

$A=\left[\begin{array}{cclr}
p_{11}&0&0\\ p_{21}&0&0\\p_{31}&-p_{31}p_{11}^{-1}r_{11}+r_{31}&1\end{array}\right]$,

$(X, Y)=\left(\left[\begin{array}{cclr}
1&0&0\\ 0&0&0\\0&p_{32}&p_{33}\end{array}\right],\left[\begin{array}{cclr}
p_{11}^{-1}r_{11}&0&0\\1&0&0\\0&r_{32}&r_{33}\end{array}\right]\right);$

\vspace*{0.3cm}

${\bf 1.1.2.}$  $p_{32}=p_{33}=r_{32}=r_{33}=0$;

\vspace*{0.3cm}

$A=\left[\begin{array}{cclr}
p_{11}&0&0\\ p_{21}&-p_{21}p_{11}^{-1}r_{11}+r_{21}&0\\p_{31}&-r_{31}p_{11}^{-1}r_{11}+r_{31}&0\end{array}\right]$,

$(X, Y)=\left(\left[\begin{array}{cclr}
1&0&0\\ 0&0&0\\0&1&0\end{array}\right],\left[\begin{array}{cclr}
p_{11}^{-1}r_{11}&0&0\\1&0&0\\0&0&0\end{array}\right]\right);$

\vspace*{0.3cm}

${\bf 1.2.}$ $p_{22}\neq 0$, $p_{33}=y_{33}=0$, $r_{31}=p_{31}p_{11}^{-1}r_{11}+p_{32}p_{22}^{-1}(r_{21}-p_{21}p_{11}^{-1}r_{11})$, $r_{32}=p_{32}p_{22}^{-1}r_{22}$;  

\vspace*{0.3cm}

$A=\left[\begin{array}{cclr}
p_{11}&0&0\\ p_{21}&1&0\\p_{31}&p_{32}p_{22}^{-1}&0\end{array}\right]$,

$(X, Y)=\left(\left[\begin{array}{cclr}
1&0&0\\ 0&p_{22}&0\\0&0&0\end{array}\right],\left[\begin{array}{cclr}
p_{11}^{-1}r_{11}&0&0\\-p_{21}p_{11}^{-1}r_{11}+r_{21}&r_{22}&0\\0&1&0\end{array}\right]\right);$

\vspace*{0.3cm}

${\bf 1.3.}$ $p_{22}=0, r_{22}\neq 0, p_{33}=r_{33}=0,
 p_{32}= 0, r_{31}= p_{31}p_{11}^{-1}r_{11}+r_{32}r_{22}^{-1}(r_{21}-$\linebreak
$-p_{21}p_{11}^{-1}r_{11})$;

\vspace*{0.3cm}

$A=\left[\begin{array}{cclr}
p_{11}&0&0\\ p_{21}&1&0\\p_{31}&r_{32}r_{22}^{-1}&0\end{array}\right]$,

$(X, Y)=\left(\left[\begin{array}{cclr}
1&0&0\\ 0&0&0\\0&1&0\end{array}\right],\left[\begin{array}{cclr}
p_{11}^{-1}r_{11}&0&0\\-p_{21}p_{11}^{-1}r_{11}+r_{21}&r_{22}&0\\0&0&0\end{array}\right]\right);$

\vspace*{0.3cm}

${\bf 2.}$ $p_{11}= 0, r_{11}\neq 0$; 
${\bf 2.1}$  $p_{22}=r_{22}=0$;

${\bf 2.1.1.}$ $p_{32}\neq 0\vee p_{33}\neq 0 \vee r_{32}\neq 0\vee r_{33}\neq 0, p_{21}=0$;  

\vspace*{0.3cm}

$A=\left[\begin{array}{cclr}
r_{11}&0&0\\ r_{21}&0&0\\r_{31}&p_{31}&1\end{array}\right]$, $(X, Y)=\left(\left[\begin{array}{cclr}
0&0&0\\ 1&0&0\\0&p_{32}&p_{33}\end{array}\right],\left[\begin{array}{cclr}
1&0&0\\0&0&0\\0&r_{32}&r_{33}\end{array}\right]\right);$

\vspace*{0.3cm}

${\bf 2.1.2.}$  $p_{32}=p_{33}=r_{32}=r_{33}=0$;
\pagebreak 

$A=\left[\begin{array}{cclr}
r_{11}&0&0\\ r_{21}&p_{21}&0\\r_{31}&p_{31}&0\end{array}\right]$, $(X, Y)=\left(\left[\begin{array}{cclr}
0&0&0\\ 1&0&0\\0&0&0\end{array}\right],\left[\begin{array}{cclr}
1&0&0\\0&0&0\\0&1&0\end{array}\right]\right);$

\vspace*{0.3cm}

${\bf 2.2.}$ $p_{22}\neq 0, p_{33}=r_{33}=0, p_{31}=p_{32}p_{22}^{-1}p_{21}, r_{32}= p_{32}p_{22}^{-1}r_{22}$; 

\vspace*{0.3cm}

$A=\left[\begin{array}{cclr}
r_{11}&0&0\\ r_{21}&1&0\\r_{31}&p_{32}p_{22}^{-1}&0\end{array}\right]$, $(X, Y)=\left(\left[\begin{array}{cclr}
0&0&0\\ p_{21}&p_{22}&0\\0&0&0\end{array}\right],\left[\begin{array}{cclr}
1&0&0\\0&r_{22}&0\\0&1&0\end{array}\right]\right);$

\vspace*{0.3cm}

${\bf 2.3.}$ $p_{22}=0, r_{22}\neq 0, p_{33}=r_{33}=0, p_{32}=0, p_{31}= r_{32}r_{22}^{-1}p_{21}$;  

\vspace*{0.3cm}

$A=\left[\begin{array}{cclr}
r_{11}&0&0\\ r_{21}&1&0\\r_{31}&r_{32}r_{22}^{-1}&0\end{array}\right]$,
$(X, Y)=\left(\left[\begin{array}{cclr}
0&0&0\\ p_{21}&0&0\\0&1&0\end{array}\right],\left[\begin{array}{cclr}
1&0&0\\0&r_{22}&0\\0&0&0\end{array}\right]\right);$

\vspace*{0.3cm}

${\bf 3.}$ $p_{11}=r_{11}=0$; 
${\bf 3.1}$  $p_{22}=r_{22}=0$;

${\bf 3.1.1.}$ $p_{32}\neq 0\vee p_{33}\neq 0 \vee r_{32}\neq 0\vee r_{33}\neq 0$;  

\vspace*{0.3cm}

$A=\left[\begin{array}{cclr}
0&0&0\\ p_{21}&r_{21}&0\\p_{31}&0&1\end{array}\right]$, $(X, Y)=\left(\left[\begin{array}{cclr}
1&0&0\\ 0&0&0\\0&p_{32}&p_{33}\end{array}\right],\left[\begin{array}{cclr}
0&0&0\\1&0&0\\r_{31}&r_{32}&r_{33}\end{array}\right]\right);$

\vspace*{0.3cm}

${\bf 3.1.2.}$  $p_{32}=p_{33}=r_{32}=r_{33}=0$;

\vspace*{0.3cm}
$A=\left[\begin{array}{cclr}
0&0&0\\ p_{21}&r_{21}&0\\p_{31}&r_{31}&0\end{array}\right]$, $(X, Y)=\left(\left[\begin{array}{cclr}
1&0&0\\ 0&0&0\\0&1&0\end{array}\right],\left[\begin{array}{cclr}
0&0&0\\1&0&0\\0&0&0\end{array}\right]\right);$

\vspace*{0.3cm}

${\bf 3.2.}$ $p_{22}\neq 0, p_{33}=r_{33}=0$; 
${\bf 3.2.1.}$ $r_{31}\neq p_{32}p_{22}^{-1}r_{21}\vee r_{32}\neq  p_{32}p_{22}^{-1}r_{22}$;  

\vspace*{0.3cm}

$A=\left[\begin{array}{cclr}
0&0&0\\ p_{21}&1&0\\p_{31}&p_{32}p_{22}^{-1}&1\end{array}\right]$,

$(X, Y)=\left(\left[\begin{array}{cclr}
1&0&0\\ 0&p_{22}&0\\0&0&0\end{array}\right],\left[\begin{array}{cclr}
0&0&0\\r_{21}&r_{22}&0\\-p_{32}p_{22}^{-1}r_{21}+r_{31}&-p_{32}p_{22}^{-1}r_{22}+r_{32}&0\end{array}\right]\right)$.

\vspace*{0.3cm}

${\bf 3.2.2.}$ $r_{31}= p_{32}p_{22}^{-1}r_{21}, r_{32}=p_{32}p_{22}^{-1}r_{22}$;  

\vspace*{0.3cm}

$A=\left[\begin{array}{cclr}
0&0&0\\ p_{21}&1&0\\p_{31}&p_{32}p_{22}^{-1}&0\end{array}\right]$, 
$(X, Y)=\left(\left[\begin{array}{cclr}
1&0&0\\ 0&p_{22}&0\\0&0&0\end{array}\right],\left[\begin{array}{cclr}
0&0&0\\r_{21}&r_{22}&0\\0&1&0\end{array}\right]\right)$;

\pagebreak

${\bf 3.3.}$ $p_{22}=0, r_{22}\neq 0, p_{33}=r_{33}=0$; 
${\bf 3.3.1.}$ $p_{32}\neq  0 \vee   r_{31}\neq r_{32}r_{22}^{-1}r_{21}$;  

\vspace*{0.3cm}

$A=\left[\begin{array}{cclr}
0&0&0\\ p_{21}&1&0\\p_{31}&r_{32}r_{22}^{-1}&1\end{array}\right]$,

$(X, Y)=\left(\left[\begin{array}{cclr}
1&0&0\\ 0&0&0\\0&p_{32}&0\end{array}\right],\left[\begin{array}{cclr}
0&0&0\\r_{21}&r_{22}&0\\-r_{32}r_{22}^{-1}r_{21}+r_{31}&0&0\end{array}\right]\right)$.

\vspace*{0.3cm}

${\bf 3.2.2.}$ $p_{32}=0, r_{31}=r_{32}r_{22}^{-1}r_{21}$;  

\vspace*{0.3cm}

$A=\left[\begin{array}{cclr}
0&0&0\\ p_{21}&1&0\\p_{31}&r_{32}r_{22}^{-1}&0\end{array}\right]$, 
$(X, Y)=\left(\left[\begin{array}{cclr}
1&0&0\\ 0&0&0\\0&1&0\end{array}\right],\left[\begin{array}{cclr}
0&0&0\\r_{21}&r_{22}&0\\0&0&0\end{array}\right]\right)$.

\end{proof}

\begin{theorem}\label{points3}
There are exactly $(q+1)^3q^3$ points of the projective line $\mathbb{P}(T_3(q))$. They can be presented as $q+1$ sets consists of $q^2+q$ subsets each: 
\begin{itemize}
\item a single set, referred to as the first set, consists of: 
\begin{itemize}
\item $q^2$ subsets of the form $$\left\{T_3(q)\big(X', Y\big), T_3(q)\big(X, Y'\big); x_{21}, x_{31}, y_{11}, y_{21}, y_{31}\in F(q)\right\},$$ where $$\big(X', Y\big)=\left(\left[\begin{array}{cclr}
1&0&0\\ 0&1&0\\0&0&1\end{array}\right],\left[\begin{array}{cclr}
y_{11}&0&0\\y_{21}&0&0\\y_{31}&y_{32}&y_{33}\end{array}\right]\right),$$
$$\big(X, Y'\big)=\left(\left[\begin{array}{cclr}
0&0&0\\ x_{21}&1&0\\x_{31}&0&1\end{array}\right],\left[\begin{array}{cclr}
1&0&0\\0&0&0\\0&y_{32}&y_{33}\end{array}\right]\right)$$ 
and $y_{32}, y_{33}$ run through all the elements of  $F(q)$;

\item $q$ subsets of the form $$\left\{T_3(q)\big(X', Y\big), T_3(q)\big(X, Y'\big); x_{21}, x_{31}, y_{11}, y_{21}, y_{31}\in F(q)\right\},$$ where $$\big(X',Y\big)=\left(\left[\begin{array}{cclr}
1&0&0\\ 0&1&0\\0&0&0\end{array}\right],\left[\begin{array}{cclr}
y_{11}&0&0\\y_{21}&0&0\\y_{31}&y_{32}&1\end{array}\right]\right),$$ 
$$\big(X, Y'\big)=\left(\left[\begin{array}{cclr}
0&0&0\\ x_{21}&1&0\\x_{31}&0&0\end{array}\right],\left[\begin{array}{cclr}
1&0&0\\0&0&0\\0&y_{32}&1\end{array}\right]\right)$$
and $y_{32}$ runs through all the elements of  $F(q)$. 
\end{itemize}

\item $q$ $k$-sets, where $k\in  F(q)$; each of them consists of:  
\begin{itemize}
\item $q^2$ subsets of the form $$\left\{T_3(q)\big(X, Y'\big), T_3(q)\big(X', Y\big);  x_{11}, x_{21}, x_{31}, y_{21}, y_{31}\in F(q)\right\},$$ where $$\big(X, Y'\big)=\left(\left[\begin{array}{cclr}
x_{11}&0&0\\x_{21}&k&0\\x_{31}&x_{32}&x_{33}\end{array}\right],\left[\begin{array}{cclr}
1&0&0\\ 0&1&0\\0&0&1\end{array}\right]\right),$$

$$\big(X', Y\big)=\left(\left[\begin{array}{cclr}
1&0&0\\0&k&0\\0&x_{32}&x_{33}\end{array}\right],\left[\begin{array}{cclr}
0&0&0\\ y_{21}&1&0\\y_{31}&0&1\end{array}\right]\right)$$ 
and $x_{32}, x_{33}$ run through all the elements of  $F(q)$;

\item $q$ subsets of the form $$\left\{T_3(q)\big(X, Y'\big), T_3(q)\big(X', Y\big); x_{11}, x_{21}, x_{31}, y_{21}, y_{31}\in F(q)\right\},$$ where $$\big(X, Y'\big)=\left(\left[\begin{array}{cclr}
x_{11}&0&0\\x_{21}&k&0\\x_{31}&x_{32}&1\end{array}\right],\left[\begin{array}{cclr}
1&0&0\\ 0&1&0\\0&0&0\end{array}\right]\right),$$
 
$$\big(X', Y\big)=\left(\left[\begin{array}{cclr}
1&0&0\\0&k&0\\0&x_{32}&1\end{array}\right],\left[\begin{array}{cclr}
0&0&0\\ y_{21}&1&0\\y_{31}&0&0\end{array}\right]\right)$$
and $x_{32}$ runs through all the elements of  $F(q)$. 
\end{itemize}\end{itemize}
\end{theorem}
\begin{proof}
The number $(q+1)^3q^3$ of unimodular free cyclic submodules of  $\mathbb{P}(T_3(q))$ follows directly from \cite[Corollary 1]{graph}.

Let $T_n(q)(X, Y), T_n(q)(W, Z)\in \mathbb{P}(T_n(q))$, then Remark \ref{generators.of.fcs} implies that 
$$T_n(q)(X, Y)=T_n(q)(W, Z)\Rightarrow\bigwedge_{\substack{x_{ii},y_{ii}\in F^*(q);\\ i=1,\dots, n}}\bigvee_{u_{ii}\in F^*(q)} u_{ii}x_{ii}=w_{ii}\wedge u_{ii}y_{ii}=z_{ii}.$$
Furthermore, $u_{22}k=l$ and $u_{22}1=1$ if, and only if, $u_{22}=1$ and $k=l$. These two facts exclude the possibility that points of $\mathbb{P}(T_3(q))$ in two distinct subsets are the same. Therefore the listed $(q+1)^2q$ subsets of points of $\mathbb{P}(T_3(q))$ are pairwise disjoint.  
Moreover, by using Remark \ref{generators.of.fcs} again we get that two points $T_3(q)(X, Y), T_3(q)(W, Z)\in \mathbb{P}(T_3(q))$ of the same subset are equal if, and only if, they are proportional by the identity matrix, equivalently $(X, Y)=(W, Z)$. So each of the listed subsets contains $q^2(q+1)$ distinct points.
\end{proof}

\begin{cor}\label{points.of.subset}
Two unimodular free cyclic submodules  $$T_3(q)\left(\left[\begin{array}{cclr}
x_{11}&0&0\\ x_{21}&x_{22}&0\\x_{31}&x_{32}&x_{33}\end{array}\right],\left[\begin{array}{cclr}
y_{11}&0&0\\y_{21}&y_{22}&0\\y_{31}&y_{32}&y_{33}\end{array}\right]\right),$$ $$T_3(q)\left(\left[\begin{array}{cclr}
w_{11}&0&0\\ w_{21}&w_{22}&0\\w_{31}&w_{32}&w_{33}\end{array}\right],\left[\begin{array}{cclr}
z_{11}&0&0\\z_{21}&z_{22}&0\\z_{31}&z_{32}&z_{33}\end{array}\right]\right)$$ of $\mathbb{P}(T_3(q))$ are in the same set of points of  $\mathbb{P}(T_3(q))$ if, and only if, $w_{22}=x_{22}, z_{22}=y_{22}$. 

Two unimodular free cyclic submodules $$T_3(q)\left(\left[\begin{array}{cclr}
x_{11}&0&0\\ x_{21}&x_{22}&0\\x_{31}&x_{32}&x_{33}\end{array}\right],\left[\begin{array}{cclr}
y_{11}&0&0\\y_{21}&y_{22}&0\\y_{31}&y_{32}&y_{33}\end{array}\right]\right),$$ $$T_3(q)\left(\left[\begin{array}{cclr}
w_{11}&0&0\\ w_{21}&x_{22}&0\\w_{31}&w_{32}&w_{33}\end{array}\right],\left[\begin{array}{cclr}
z_{11}&0&0\\z_{21}&y_{22}&0\\z_{31}&z_{32}&z_{33}\end{array}\right]\right)$$ of ${^2T_3(q)}$   are in the same  subset of points of  $\mathbb{P}(T_3(q))$ if, and only if, $w_{32}=x_{32}, w_{33}=x_{33}, z_{32}=y_{32}, z_{33}=y_{33}$. 
\end{cor}

\begin{prop}\label{nfcs.ufcs.3}
Let  
$$T_3(q)\big(P, R\big)=T_3(q)\left(\left[\begin{array}{cclr}
0&0&0\\ p_{21}&p_{22}&0\\p_{31}&p_{32}&p_{33}\end{array}\right],\left[\begin{array}{cclr}
0&0&0\\r_{21}&r_{22}&0\\r_{31}&r_{32}&r_{33}\end{array}\right]\right)\subset {^2T_3(q)},$$ be a non-free cyclic submodule  not contained in any non-unimodular free cyclic submodule of ${^2T_3(q)}$, and let 
$$T_3(q)\big(X, Y\big)=T_3(q)\left(\left[\begin{array}{cclr}
x_{11}&0&0\\ x_{21}&x_{22}&0\\x_{31}&x_{32}&x_{33}\end{array}\right],\left[\begin{array}{cclr}
y_{11}&0&0\\y_{21}&y_{22}&0\\y_{31}&y_{32}&y_{33}\end{array}\right]\right)\in \mathbb{P}(T_3(q)),$$
where $\big(X, Y\big)$ is of the form as in Theorem \ref{points3}.

$T_3(q)\big(P, R\big)\subset T_3(q)\big(X, Y\big)$ if, and only if, all of the following conditions are met: 
\begin{enumerate}
\item
$x_{22}=p_{22}, x_{32}=p_{32},  x_{33}=p_{33}, y_{22}=r_{22}, y_{32}=r_{32},  y_{33}=r_{33}$;
\item
$x_{21}=p_{21}-r_{21}x_{11}, x_{31}=p_{31}-r_{31}x_{11}$ if $y_{11}=1, y_{21}=y_{31}=0$ ; 
\item $y_{21}=r_{21}-p_{21}y_{11}, y_{31}=r_{31}-p_{31}y_{11}$ if $x_{11}=1, x_{21}=x_{31}=0$.
\end{enumerate}
\end{prop}
\begin{proof}
Let 
$(X, Y)\subset {^2T_3(q)}$ be a unimodular pair. Suppose that $$T_3(q)\big(P, R\big)\subset T_3(q)\big(X, Y\big).$$ 
Equivalently, there exists a matrix $A=\left[\begin{array}{cclr}
a_{11}&0&0\\ a_{21}&a_{22}&0\\a_{31}&a_{32}&a_{33}\end{array}\right]\in T_3(q)$ such that $A(X, Y)= \left(P, R\right).$ 

If $T_3(q)\big(P, R\big)=T_3(q)\left(\left[\begin{array}{cclr}
0&0&0\\ p_{21}&1&0\\p_{31}&0&1\end{array}\right],\left[\begin{array}{cclr}
0&0&0\\r_{21}&0&0\\r_{31}&r_{32}&r_{33}\end{array}\right]\right)$, 
then $y_{22}=0,$\linebreak $x_{22}, x_{33}\in F^*(q)$ and $a_{22}=x_{22}^{-1}, a_{33}=x_{33}^{-1}$. By Theorem \ref{points3} we can put on $x_{22}=x_{33}=1, x_{32}=0$. Therefore $r_{32}=a_{32}y_{22}+a_{33}y_{32}=y_{32}, r_{33}=a_{33}y_{33}=y_{33}$.  

In the same manner it can be shown  that $x_{22}=p_{22}, x_{32}=p_{32},  x_{33}=p_{33}, y_{22}=r_{22}, y_{32}=r_{32},  y_{33}=r_{33}$ in all other cases   of non-free cyclic submodules of ${^2T_3(q)}$ not contained in any non-unimodular free cyclic submodule of ${^2T_3(q)}$. 

In the light of Theorem \ref{points3} we can consider two cases of $T_3(q)\big(X, Y\big)$:
\begin{enumerate}
\item $x_{11}=1, x_{21}=x_{31}=0$; 

Then $A(X, Y)= \left(P, R\right)\Leftrightarrow a_{11}=a_{32}=0, a_{22}=a_{33}=1, a_{21}=p_{21}, a_{31}=p_{31}, y_{21}=r_{21}-p_{21}y_{11}, y_{31}=r_{31}-p_{31}y_{11}.$
\item $y_{11}=1, y_{21}=y_{31}=0$.

Then $A(X, Y)= \left(P, R\right)\Leftrightarrow a_{11}=a_{32}=0, a_{22}=a_{33}=1, a_{21}=r_{21}, a_{31}=r_{31}, x_{21}=p_{21}-r_{21}x_{11}, x_{31}=p_{31}-r_{31}x_{11}.$
\end{enumerate}
\end{proof}

\begin{cor}\label{parallel.subset}
Let 
$$T_3(q)\big(P, R\big)=T_3(q)\left(\left[\begin{array}{cclr}
0&0&0\\ p_{21}&p_{22}&0\\p_{31}&p_{32}&p_{33}\end{array}\right],\left[\begin{array}{cclr}
0&0&0\\r_{21}&r_{22}&0\\r_{31}&r_{32}&r_{33}\end{array}\right]\right)\subset {^2T_3(q)},$$ be a non-free cyclic submodule  not contained in any non-unimodular free cyclic submodule of ${^2T_3(q)}$.

Two elements 
$$T_3(q)\big(X, Y\big)=T_3(q)\left(\left[\begin{array}{cclr}
x_{11}&0&0\\ x_{21}&x_{22}&0\\x_{31}&x_{32}&x_{33}\end{array}\right],\left[\begin{array}{cclr}
y_{11}&0&0\\y_{21}&y_{22}&0\\y_{31}&y_{32}&y_{33}\end{array}\right]\right),$$
$$T_3(q)\big(X', Y'\big)=T_3(q)\left(\left[\begin{array}{cclr}
x'_{11}&0&0\\ x'_{21}&x_{22}&0\\x'_{31}&x_{32}&x_{33}\end{array}\right],\left[\begin{array}{cclr}
y'_{11}&0&0\\y'_{21}&y_{22}&0\\y'_{31}&y_{32}&y_{33}\end{array}\right]\right)$$ 
of the fixed subset of points of $\mathbb{P}(T_3(q))$, where $\big(X, Y\big), \big(X', Y'\big)$ are of the form as in Theorem \ref{points3}, have no submodule $T_3(q)\big(P, R\big)$ in common if, and only if,
\begin{itemize}
\item
$X=X'=\left[\begin{array}{cclr}
1&0&0\\ 0&x_{22}&0\\0&x_{32}&x_{33}\end{array}\right], Y'=\left[\begin{array}{cclr}
y_{11}&0&0\\y'_{21}&y_{22}&0\\y'_{31}&y_{32}&y_{33}\end{array}\right]$, and $y'_{21}\neq y_{21}\vee y'_{31}\neq y_{31}$ or 
\item
$Y=Y'=\left[\begin{array}{cclr}
1&0&0\\ 0&y_{22}&0\\0&y_{32}&y_{33}\end{array}\right], X'=\left[\begin{array}{cclr}
x_{11}&0&0\\x'_{21}&x_{22}&0\\x'_{31}&x_{32}&x_{33}\end{array}\right]$, and $x'_{21}\neq x_{21}\vee x'_{31}\neq x_{31}.$
\end{itemize}
\end{cor}

\begin{ex}\label{contain3}
Let $y_{32}, y_{33}$ be fixed elements of  $F(q)$ and $A=\left[\begin{array}{cclr}
a_{11}&0&0\\ a_{21}&a_{22}&0\\a_{31}&a_{32}&a_{33}\end{array}\right]$ be a matrix of $T_3(q)$.
Choose the subset of points of $\mathbb{P}(T_3(q))$ of the form 
$$\left\{T_3(q)\big(X', Y\big), T_3(q)\big(X, Y'\big); x_{21}, x_{31}, y_{11}, y_{21}, y_{31}\in F(q)\right\},$$ where $$\big(X', Y\big)=\left(\left[\begin{array}{cclr}
1&0&0\\ 0&1&0\\0&0&1\end{array}\right],\left[\begin{array}{cclr}
y_{11}&0&0\\y_{21}&0&0\\y_{31}&y_{32}&y_{33}\end{array}\right]\right),$$
$$\big(X, Y'\big)=\left(\left[\begin{array}{cclr}
0&0&0\\ x_{21}&1&0\\x_{31}&0&1\end{array}\right],\left[\begin{array}{cclr}
1&0&0\\0&0&0\\0&y_{32}&y_{33}\end{array}\right]\right).$$ 
$A\big(X', Y\big)=\left(\left[\begin{array}{cclr}
0&0&0\\ p_{21}&1&0\\p_{31}&0&1\end{array}\right],\left[\begin{array}{cclr}
0&0&0\\r_{21}&0&0\\r_{31}&y_{32}&y_{33}\end{array}\right]\right)\Leftrightarrow a_{11}=a_{32}=0, a_{22}=$ $=a_{33}=1, a_{21}=p_{21}, a_{31}=p_{31}, r_{21}=p_{21}y_{11}+y_{21}, r_{31}=p_{31}y_{11}+y_{31}$ and $p_{21}, p_{31}\in F(q)$.
 
$A\big(X, Y'\big)=\left(\left[\begin{array}{cclr}
0&0&0\\ p_{21}&1&0\\p_{31}&0&1\end{array}\right],\left[\begin{array}{cclr}
0&0&0\\r_{21}&0&0\\r_{31}&y_{32}&y_{33}\end{array}\right]\right)\Leftrightarrow a_{11}=a_{32}=0, a_{22}=$ $=a_{33}=1, a_{21}=r_{21}, a_{31}=r_{31}, p_{21}=x_{21},  p_{31}=x_{31}$ and $r_{21}, r_{31}\in F(q)$.

Hence
$$T_3(q)\left(\left[\begin{array}{cclr}
0&0&0\\ p_{21}&1&0\\p_{31}&0&1\end{array}\right],\left[\begin{array}{cclr}
0&0&0\\p_{21}y_{11}+y_{21}&0&0\\p_{31}y_{11}+y_{31}&y_{32}&y_{33}\end{array}\right]\right)\subset {^2T_3(q)}\big(X', Y\big),$$ for any $p_{21}, p_{31}, y_{11}, y_{21}, y_{31}, y_{32}, y_{33}\in F(q),$ and

$$T_3(q)\left(\left[\begin{array}{cclr}
0&0&0\\ x_{21}&1&0\\x_{31}&0&1\end{array}\right],\left[\begin{array}{cclr}
0&0&0\\r_{21}&0&0\\r_{31}&y_{32}&y_{33}\end{array}\right]\right)\subset {^2T_3(q)}\big(X, Y'\big)$$ for any $x_{21}, x_{31}, r_{21}, r_{31}, y_{32}, y_{33}\in F(q)$. According to Proposition \ref{nfcs.ufcs.3} these are the only non-free cyclic submodules of ${^2T_3(q)}$ not contained in any non-unimodular free cyclic submodule of ${^2T_3(q)}$  and   contained in the considered points of $\mathbb{P}(T(q))$.
\end{ex}

\begin{theorem}\label{affine.plane.3}
Let 
$$\mathbb{A}=\left\{T_3(q)\left(\left[\begin{array}{cclr}
0&0&0\\ p_{21}&x_{22}&0\\p_{31}&x_{32}&x_{33}\end{array}\right],\left[\begin{array}{cclr}
0&0&0\\r_{21}&y_{22}&0\\r_{31}&y_{32}&y_{33}\end{array}\right]\right); p_{31}, r_{31}\in F(q)\right\},$$ where  $p_{21}, x_{22}, x_{32}, x_{33}, r_{21}, y_{22}, y_{32}, y_{33}$ are fixed elements of $F(q)$, be contained in a subset
of non-free cyclic submodules not contained in any non-unimodular free cyclic submodules, and let us regard its elements   as points and free cyclic submodules  containing them as lines. 

Then these points and lines form a point-line incidence structure isomorphic to the {\it affine} plane of order $q$.
\end{theorem}
\begin{proof}
Let $$T_3(q)\big(P, R\big)=T_3(q)\left(\left[\begin{array}{cclr}
0&0&0\\ p_{21}&x_{22}&0\\p_{31}&x_{32}&x_{33}\end{array}\right],\left[\begin{array}{cclr}
0&0&0\\r_{21}&y_{22}&0\\r_{31}&y_{32}&y_{33}\end{array}\right]\right)$$ and $$T_3(q)\big(S, V\big)=T_3(q)\left(\left[\begin{array}{cclr}
0&0&0\\ s_{21}&x_{22}&0\\s_{31}&x_{32}&x_{33}\end{array}\right],\left[\begin{array}{cclr}
0&0&0\\v_{21}&y_{22}&0\\v_{31}&y_{32}&y_{33}\end{array}\right]\right)$$ be any two distinct submodules of $\mathbb{A}$. Assume that they both are contained in the point $T_3(q)\left(\left[\begin{array}{cclr}
1&0&0\\ 0&x_{22}&0\\0&x_{32}&x_{33}\end{array}\right],\left[\begin{array}{cclr}
y_{11}&0&0\\y_{21}&y_{22}&0\\y_{31}&y_{32}&y_{33}\end{array}\right]\right)\in \mathbb{P}(T_3(q))$. By Proposition 
\ref{nfcs.ufcs.3}  we obtain:
$r_{21}=p_{21}y_{11}+y_{21},$
$r_{31}=p_{31}y_{11}+y_{31},$ 
$v_{21}=s_{21}y_{11}+y_{21},$
$v_{31}=s_{31}y_{11}+y_{31}.$
Hence  $(r_{31}-v_{31})=(p_{31}-s_{31})y_{11}$. $p_{31}\neq s_{31}$, so $y_{11}=(p_{31}-s_{31})^{-1}(r_{31}-v_{31})$. Therefore, $y_{11}$ is uniquely determined and so there are $y_{21}, y_{31}$.
Similarly, we get that  $x_{11}, x_{21}, x_{31}$ are also uniquely determined in case of submodules $T_3(q)(P, R), T_3(q)(S, V)$ contained in the point\linebreak $T_3(q)(X, Y)=T_3(q)\left(\left[\begin{array}{cclr}
x_{11}&0&0\\ x_{21}&x_{22}&0\\x_{31}&x_{32}&x_{33}\end{array}\right],\left[\begin{array}{cclr}
1&0&0\\0&y_{22}&0\\0&y_{32}&y_{33}\end{array}\right]\right)\in \mathbb{P}(T_3(q))$.

We have just shown that any two distinct submodules of $\mathbb{A}$ are contained in exactly one point of $\mathbb{P}(T_3(q))$. Thereby axiom {\bf A1} of Definition \ref{aff} is satisfied.

\vspace*{0.3cm}

Let $T_3(q)(X, Y)=T_3(q)\left(\left[\begin{array}{cclr}
1&0&0\\ 0&x_{22}&0\\0&x_{32}&x_{33}\end{array}\right],\left[\begin{array}{cclr}
y_{11}&0&0\\y_{21}&y_{22}&0\\y_{31}&y_{32}&y_{33}\end{array}\right]\right)$ be a fixed element of the set $L$ of all free cyclic submodules containing elements of $\mathbb{A}$.
From Proposition 
\ref{nfcs.ufcs.3}  we get:
\begin{itemize}
\item an element of $\mathbb{A}$ not contained in the point $T_3(q)(X, Y)$ is of the form 
$$T_3(q)(P, R')=T_3(q)\left(\left[\begin{array}{cclr}
0&0&0\\ p_{21}&x_{22}&0\\p_{31}&x_{32}&x_{33}\end{array}\right],\left[\begin{array}{cclr}
0&0&0\\p_{21}y_{11}+y_{21}&y_{22}&0\\p_{31}y_{11}+y'_{31}&y_{32}&y_{33}\end{array}\right]\right)$$ 
where  $y'_{31}\in F(q)$ and $y'_{31}\neq y_{31}$;
\item a point of $L$ not containing any element of $\mathbb{A}$ contained in  $T_3(q)(X, Y)$ is of the form $$T_3(q)(X, Y'')=T_3(q)\left(\left[\begin{array}{cclr}
1&0&0\\ 0&x_{22}&0\\0&x_{32}&x_{33}\end{array}\right],\left[\begin{array}{cclr}
y_{11}&0&0\\y_{21}&y_{22}&0\\y''_{31}&y_{32}&y_{33}\end{array}\right]\right),$$ where  $y''_{31}\in F(q)$ and $y''_{31}\neq y_{31}$;
\item there is a unique point of the form $T_3(q)(X, Y'')$, which contains a fixed submodule $T_3(q)(P, R')$, namely 
$$T_3(q)(X, Y')=T_3(q)\left(\left[\begin{array}{cclr}
1&0&0\\ 0&x_{22}&0\\0&x_{32}&x_{33}\end{array}\right],\left[\begin{array}{cclr}
y_{11}&0&0\\y_{21}&y_{22}&0\\y'_{31}&y_{32}&y_{33}\end{array}\right]\right).$$
\end{itemize}
We have just proved that axiom {\bf A2} of Definition \ref{aff} is satisfied for any \linebreak 
$$T_3(q)\left(\left[\begin{array}{cclr}
1&0&0\\ 0&x_{22}&0\\0&x_{32}&x_{33}\end{array}\right],\left[\begin{array}{cclr}
y_{11}&0&0\\y_{21}&y_{22}&0\\y_{31}&y_{32}&y_{33}\end{array}\right]\right)\in \mathbb{P}(T_3(q)).$$ Of course, the result is the same for any free cyclic submodule\linebreak  $$T_3(q)\left(\left[\begin{array}{cclr}
x_{11}&0&0\\ x_{21}&x_{22}&0\\x_{31}&x_{32}&x_{33}\end{array}\right],\left[\begin{array}{cclr}
1&0&0\\0&y_{22}&0\\0&y_{32}&y_{33}\end{array}\right]\right)\in \mathbb{P}(T_3(q)).$$

\vspace*{0.3cm}

Assume that submodules 
$T_3(q)\Bigg(\left[\begin{array}{cclr}
0&0&0\\ 0&x_{22}&0\\p_{31}&x_{32}&x_{33}\end{array}\right],\left[\begin{array}{cclr}
0&0&0\\ 0&y_{22}&0\\r_{31}&y_{32}&y_{33}\end{array}\right]\Bigg)$ of any set $\mathbb{A}$ such that $p_{31}=0, r_{31}=0$ (the first one), $p_{31}=0, r_{31}=1$ (the second one),  and
$p_{31}=1,
r_{31}=0$ (the third one) 
 are contained in the same free cyclic submodule $T_3(q)\left(\left[\begin{array}{cclr}
x_{11}&0&0\\ x_{21}&x_{22}&0\\x_{31}&x_{32}&x_{33}\end{array}\right],\left[\begin{array}{cclr}
y_{11}&0&0\\y_{21}&y_{22}&0\\y_{31}&y_{32}&y_{33}\end{array}\right]\right)\in  \mathbb{P}(T_3(q))$. Then there exist  $a_{31}, b_{31}, c_{31}\in F(q)$ such that
$$\begin{cases}
1.\ a_{31}x_{11}+x_{31}=0.\\2.\ a_{31}y_{11}+y_{31}=0.\\3.\ b_{31}x_{11}+x_{31}=0.\\4.\ b_{31}y_{11}+y_{31}=1.\\5.\ c_{31}x_{11}+x_{31}=1.\\6.\ c_{31}y_{11}+y_{31}=0.\end{cases}$$ Hence $(b_{31}-a_{31})x_{11}=0, (b_{31}-a_{31})y_{11}=1$. Consequently $b_{31}\neq a_{31}, x_{11}=0$ and $x_{31}=-a_{31}x_{11}=0$. But then we get $c_{31}x_{11}+x_{31}=0$, which contradicts   equation (5) and thereby it contradicts the assumption that the above-given three  submodules are contained in the same point of  $\mathbb{P}(T_3(q))$.

It means that for any set $\mathbb{A}$ there exist three submodules  not contained in the same point of  $\mathbb{P}(T_3(q))$, thereby axiom {\bf A3} of Definition \ref{aff} is satisfied.

\vspace*{0.3cm}

Moreover, any element of $L$ contains exactly $q$ submodules of $\mathbb{A}$, what follows directly from Proposition \ref{nfcs.ufcs.3}.  

\vspace*{0.3cm}

This shows that any set $\mathbb{A}$  and the points of  $\mathbb{P}(T_3(q))$ containing elements of $\mathbb{A}$  give a point-line incidence structure isomorphic to the affine plane of order $q$. So, there are altogether $q^3(q+1)^2$ mutually isomorphic affine planes of order $q$ associated with the projective line  $\mathbb{P}(T_3(q))$.
\end{proof}

The above considerations lead to the following results:

\begin{cor}\label{points.subset.affine.parallel}\begin{enumerate}
\item\label{points.subset}
Two unimodular free cyclic submodules  $$T_3(q)\left(\left[\begin{array}{cclr}
x_{11}&0&0\\ x_{21}&x_{22}&0\\x_{31}&x_{32}&x_{33}\end{array}\right],\left[\begin{array}{cclr}
y_{11}&0&0\\y_{21}&y_{22}&0\\y_{31}&y_{32}&y_{33}\end{array}\right]\right),$$ $$T_3(q)\left(\left[\begin{array}{cclr}
w_{11}&0&0\\ w_{21}&w_{22}&0\\w_{31}&w_{32}&w_{33}\end{array}\right],\left[\begin{array}{cclr}
z_{11}&0&0\\z_{21}&z_{22}&0\\z_{31}&z_{32}&z_{33}\end{array}\right]\right)$$ of $\mathbb{P}(T_3(q))$ represent lines of affine planes of the same subset if, and only if, 
$w_{22}=x_{22}, w_{32}=x_{32}, w_{33}=x_{33}, z_{22}=y_{22}, z_{32}=y_{32}, z_{33}=y_{33}$. 

This follows from {\rm Corollary \ref{points.of.subset}} and  {\rm Theorem \ref{affine.plane.3}}.

\item\label{points.affine} Two unimodular free cyclic submodules  $$T_3(q)\big(X, Y\big)=T_3(q)\left(\left[\begin{array}{cclr}
x_{11}&0&0\\ x_{21}&x_{22}&0\\x_{31}&x_{32}&x_{33}\end{array}\right],\left[\begin{array}{cclr}
y_{11}&0&0\\y_{21}&y_{22}&0\\y_{31}&y_{32}&y_{33}\end{array}\right]\right),$$ 
$$T_3(q)\big(W, Z\big)=T_3(q)\left(\left[\begin{array}{cclr}
w_{11}&0&0\\ w_{21}&x_{22}&0\\w_{31}&x_{32}&x_{33}\end{array}\right],\left[\begin{array}{cclr}
z_{11}&0&0\\z_{21}&y_{22}&0\\z_{31}&y_{32}&y_{33}\end{array}\right]\right)$$ of $\mathbb{P}(T_3(q))$, where $\big(X, Y\big), \big(W, Z\big)$ are of the form as in Theorem \ref{points3},  represent lines of the same affine plane if, and only if, there exist non-free cyclic submodules 
$$T_3(q)\left(\left[\begin{array}{cclr}
0&0&0\\ s_{21}&x_{22}&0\\s_{31}&x_{32}&x_{33}\end{array}\right],\left[\begin{array}{cclr}
0&0&0\\s'_{21}&y_{22}&0\\s'_{31}&y_{32}&y_{33}\end{array}\right]\right),$$
$$T_3(q)\left(\left[\begin{array}{cclr}
0&0&0\\ s_{21}&x_{22}&0\\v_{31}&x_{32}&x_{33}\end{array}\right],\left[\begin{array}{cclr}
0&0&0\\s'_{21}&y_{22}&0\\v'_{31}&y_{32}&y_{33}\end{array}\right]\right)$$
 of ${^2T_3(q)}$  not contained in any non-unimodular free cyclic submodule of ${^2T_3(q)}$ such that:

$x_{21}=p_{21}-r_{21}x_{11}, x_{31}=p_{31}-r_{31}x_{11}$ if $y_{11}=1, y_{21}=y_{31}=0$,

$y_{21}=r_{21}-p_{21}y_{11}, y_{31}=r_{31}-p_{31}y_{11}$ if $x_{11}=1, x_{21}=x_{31}=0$,

$w_{21}=p_{21}-r_{21}w_{11}, w_{31}=s_{31}-v_{31}w_{11}$  if $z_{11}=1, z_{21}=z_{31}=0$, 

$z_{21}=r_{21}-p_{21}z_{11}, z_{31}=v_{31}-s_{31}z_{11}$  if $w_{11}=1, w_{21}=w_{31}=0$. 

This follows from {\rm Proposition \ref{nfcs.ufcs.3}} and  {\rm Theorem \ref{affine.plane.3}}.

\item\label{points.parallel} Two unimodular free cyclic submodules  $$T_3(q)\big(X, Y\big)=T_3(q)\left(\left[\begin{array}{cclr}
x_{11}&0&0\\ x_{21}&x_{22}&0\\x_{31}&x_{32}&x_{33}\end{array}\right],\left[\begin{array}{cclr}
y_{11}&0&0\\y_{21}&y_{22}&0\\y_{31}&y_{32}&y_{33}\end{array}\right]\right),$$ 
$$T_3(q)\big(W, Z\big)=T_3(q)\left(\left[\begin{array}{cclr}
w_{11}&0&0\\ w_{21}&x_{22}&0\\w_{31}&x_{32}&x_{33}\end{array}\right],\left[\begin{array}{cclr}
z_{11}&0&0\\z_{21}&y_{22}&0\\z_{31}&y_{32}&y_{33}\end{array}\right]\right)$$ of $\mathbb{P}(T_3(q))$ represent parallel lines of an affine plane if, and only if, 
$$T_3(q)\big(W, Z\big)=T_3(q)\left(\left[\begin{array}{cclr}
x_{11}&0&0\\ x_{21}&x_{22}&0\\w_{31}&x_{32}&x_{33}\end{array}\right],\left[\begin{array}{cclr}
y_{11}&0&0\\y_{21}&y_{22}&0\\z_{31}&y_{32}&y_{33}\end{array}\right]\right).$$
They are distinct exactly if, $w_{31}\neq x_{31}$ or  $z_{31}\neq y_{31}$.

This follows from {\rm Corollary  \ref{parallel.subset}} and  {\rm Theorem \ref{affine.plane.3}}.    
\end{enumerate}
\end{cor}

\begin{theorem}
Any affine plane of order $q$ associated with $\mathbb{P}(T_3(q))$ can be extended to the projective plane of order $q$ in the following way. 

Consider the set of all submodules
$$T_3(q)\left(\left[\begin{array}{cclr}
x_{11}&0&0\\ x_{21}&x_{22}&0\\x_{31}&x_{32}&x_{33}\end{array}\right],\left[\begin{array}{cclr}
y_{11}&0&0\\y_{21}&y_{22}&0\\y_{31}&y_{32}&y_{33}\end{array}\right]\right)\in  \mathbb{P}(T_3(q))$$ representing lines of a given affine plane associated with $\mathbb{P}(T_3(q))$. 
\begin{enumerate}
\item For all such submodules, where $x_{11}, x_{21}, y_{11}, y_{21}$ are fixed elements of $F(q)$, that is to say for fixed set of parallel lines of an affine plane,  a submodule  $$T_3(q)\left(\left[\begin{array}{cclr}
0&0&0\\ 0&0&0\\x_{11}&0&0\end{array}\right],\left[\begin{array}{cclr}
0&0&0\\0&0&0\\y_{11}&0&0\end{array}\right]\right)\subset {^2T_3(q)}$$ must be taken into account as a new point, and
\item A  free cyclic submodule  $$T_3(q)\left(\left[\begin{array}{cclr}
x_{22}&0&0\\ y_{22}&0&0\\0&\delta_{y_{22}0}&0\end{array}\right],\left[\begin{array}{cclr}
y_{22}&0&0\\\delta_{y_{22}0}&0&0\\0&y_{22}&0\end{array}\right]\right)\subset {^2T_3(q)},$$ where $\delta_{y_{22}0}$ stands for the Kronecker delta, must be taken into account as a new line. 
\end{enumerate} 
\end{theorem}
\begin{proof} Obviously, $T_3(q)\left(\left[\begin{array}{cclr}
x_{11}&0&0\\ x_{21}&x_{22}&0\\x_{31}&x_{32}&x_{33}\end{array}\right],\left[\begin{array}{cclr}
y_{11}&0&0\\y_{21}&y_{22}&0\\y_{31}&y_{32}&y_{33}\end{array}\right]\right)\in  \mathbb{P}(T_3(q))$
contains 
$T_3(q)\left(\left[\begin{array}{cclr}
0&0&0\\ 0&0&0\\x_{11}&0&0\end{array}\right],\left[\begin{array}{cclr}
0&0&0\\ 0&0&0\\y_{11}&0&0\end{array}\right]\right)$. 

By Corollary \ref{points.subset.affine.parallel} (\ref{points.affine} and \ref{points.parallel}) we get immediately that if unimodular free cyclic submodules $$T_3(q)\left(\left[\begin{array}{cclr}
x_{11}&0&0\\ x_{21}&x_{22}&0\\x_{31}&x_{32}&x_{33}\end{array}\right],\left[\begin{array}{cclr}
y_{11}&0&0\\y_{21}&y_{22}&0\\y_{31}&y_{32}&y_{33}\end{array}\right]\right),$$ $$T_3(q)\left(\left[\begin{array}{cclr}
w_{11}&0&0\\ w_{21}&x_{22}&0\\w_{31}&x_{32}&x_{33}\end{array}\right],\left[\begin{array}{cclr}
z_{11}&0&0\\z_{21}&y_{22}&0\\z_{31}&y_{32}&y_{33}\end{array}\right]\right)$$ represent non-parallel lines of an affine plane, then $x_{11}\neq w_{11}$ or $y_{11}\neq z_{11}$. According to Theorem \ref{points3} we can assume that $x_{11}=w_{11}=1$ or $y_{11}=z_{11}=1$. Consequently each of cyclic submodules $\left(\left[\begin{array}{cclr}
0&0&0\\ 0&0&0\\x_{11}&0&0\end{array}\right],\left[\begin{array}{cclr}
0&0&0\\ 0&0&0\\y_{11}&0&0\end{array}\right]\right)$,
$\left(\left[\begin{array}{cclr}
0&0&0\\ 0&0&0\\w_{11}&0&0\end{array}\right],\left[\begin{array}{cclr}
0&0&0\\ 0&0&0\\z_{11}&0&0\end{array}\right]\right)$ of ${^2T_3(q)}$
contained in distinct sets of parallel lines of an affine plane satisfies one of the two conditions: $x_{11}=w_{11}=1, y_{11}\neq z_{11}$ or $y_{11}=z_{11}=1, x_{11}\neq w_{11}$, hence they are also distinct. 
We have shown that any line of a set of parallel lines of an affine plane contains the added new point and these new points are distinct for distinct sets of parallel lines.

By Theorem \ref{fcs.out.3} it suffices to consider only two cases of non-unimodular free cyclic submodules $T_3(q)\left(\left[\begin{array}{cclr}
x_{22}&0&0\\ y_{22}&0&0\\0&\delta_{y_{22}0}&0\end{array}\right],\left[\begin{array}{cclr}
y_{22}&0&0\\\delta_{y_{22}0}&0&0\\0&y_{22}&0\end{array}\right]\right)\subset {^2T_3(q)}$, namely 
$x_{22}=1, y_{22}=0$ and 
$x_{22}\in F(q), y_{22}=1$. 
It is easy to verify now that for any affine plane the added new line contains all the new points.
\end{proof}

\begin{theorem}
Let us regard sets of parallel lines of affine planes associated with a subset of points of $\mathbb{P}(T_3(q))$ as lines, and sets of points of such affine planes as points. Then these points and lines form a point-line incidence structure, called a {\sl $2$-affine plane}, isomorphic to the affine plane of order $q$. \linebreak
A point and a line of a $2$-affine plane are incident if any element of the set representing this point is contained in some element of the set  representing this line.
\end{theorem}
\begin{proof}
According to Corollary \ref{points.subset.affine.parallel} (\ref{points.parallel}) and Theorem \ref{points3}, sets of parallel lines of  affine planes associated with a subset of points of $\mathbb{P}(T_3(q))$, i.e., lines of a 2-affine plane, are one of the two following forms:
$$\left\{T_3(q)\left(\left[\begin{array}{cclr}
1&0&0\\ 0&x_{22}&0\\0&x_{32}&x_{33}\end{array}\right],\left[\begin{array}{cclr}
y_{11}&0&0\\y_{21}&y_{22}&0\\y_{31}&y_{32}&y_{33}\end{array}\right]\right); y_{31}\in F(q)\right\},$$
$$\left\{T_3(q)\left(\left[\begin{array}{cclr}
0&0&0\\ x_{21}&x_{22}&0\\x_{31}&x_{32}&x_{33}\end{array}\right],\left[\begin{array}{cclr}
1&0&0\\0&y_{22}&0\\0&y_{32}&y_{33}\end{array}\right]\right); x_{31}\in F(q)\right\},$$
where $x_{21}, y_{11}, y_{21}$ run through all the elelments of $F(q)$, $x_{22}=1, y_{22}=0$ and $x_{32}, x_{33}, y_{32}, y_{33}$ are fixed elements of $F(q)$ such that $x_{33}=1\vee y_{33}=1;$

\vspace*{0.3cm}

or

$$\left\{T_3(q)\left(\left[\begin{array}{cclr}
x_{11}&0&0\\ x_{21}&x_{22}&0\\x_{31}&x_{32}&x_{33}\end{array}\right],\left[\begin{array}{cclr}
1&0&0\\0&y_{22}&0\\0&y_{32}&y_{33}\end{array}\right]\right); x_{31}\in F(q)\right\},$$
$$\left\{T_3(q)\left(\left[\begin{array}{cclr}
1&0&0\\ 0&x_{22}&0\\0&x_{32}&x_{33}\end{array}\right],\left[\begin{array}{cclr}
0&0&0\\y_{21}&y_{22}&0\\y_{31}&y_{32}&y_{33}\end{array}\right]\right); y_{31}\in F(q)\right\},$$
where $x_{11}, x_{21}, y_{21}$ run through all the elelments of $F(q)$, $y_{22}=1$ and $x_{22}$, $x_{32}$, $x_{33}$, $y_{32}$, $y_{33}$ are fixed elements of $F(q)$ such that $x_{33}=1\vee y_{33}=1.$

Then sets of points of such affine planes, i.e., points of a 2-affine plane, are of the form
$$\left\{T_3(q)\left(\left[\begin{array}{cclr}
0&0&0\\ p_{21}&x_{22}&0\\p_{31}&x_{32}&x_{33}\end{array}\right],\left[\begin{array}{cclr}
0&0&0\\r_{21}&y_{22}&0\\r_{31}&y_{32}&y_{33}\end{array}\right]\right); p_{31}, r_{31}\in F(q)\right\},$$
where $p_{21}, r_{21}$ run through all the elelments of $F(q)$.

\vspace*{0.3cm}

Consider two distinct points 
$$P=\left\{T_3(q)\left(\left[\begin{array}{cclr}
0&0&0\\ p_{21}&x_{22}&0\\p_{31}&x_{32}&x_{33}\end{array}\right],\left[\begin{array}{cclr}
0&0&0\\r_{21}&y_{22}&0\\r_{31}&y_{32}&y_{33}\end{array}\right]\right); p_{31}, r_{31}\in F(q)\right\},$$

$$S=\left\{T_3(q)\left(\left[\begin{array}{cclr}
0&0&0\\ s_{21}&x_{22}&0\\s_{31}&x_{32}&x_{33}\end{array}\right],\left[\begin{array}{cclr}
0&0&0\\v_{21}&y_{22}&0\\v_{31}&y_{32}&y_{33}\end{array}\right]\right); s_{31}, v_{31}\in F(q)\right\}.$$
By Proposition \ref{nfcs.ufcs.3} they are inicident with a line $l$ if, and only if, 

$$l=l_1=\left\{T_3(q)\left(\left[\begin{array}{cclr}
1&0&0\\ 0&x_{22}&0\\0&x_{32}&x_{33}\end{array}\right],\left[\begin{array}{cclr}
y_{11}&0&0\\y_{21}&y_{22}&0\\y_{31}&y_{32}&y_{33}\end{array}\right]\right); y_{31}\in F(q)\right\},$$
where $y_{21}=r_{21}-p_{21}y_{11}=v_{21}-s_{21}y_{11}$. Then $r_{21}-v_{21}=(p_{21}-s_{21})y_{11}$;

or

$$l=l_2=\left\{T_3(q)\left(\left[\begin{array}{cclr}
x_{11}&0&0\\ x_{21}&x_{22}&0\\x_{31}&x_{32}&x_{33}\end{array}\right],\left[\begin{array}{cclr}
1&0&0\\0&y_{22}&0\\0&y_{32}&y_{33}\end{array}\right]\right); x_{31}\in F(q)\right\},$$
where $x_{21}=p_{21}-r_{21}x_{11}=s_{21}-v_{21}x_{11}$; hence, we get $p_{21}-s_{21}=(r_{21}-v_{21})x_{11}.$

$P\neq S\Leftrightarrow p_{21}\neq s_{21}\vee r_{21}\neq v_{21}$, hence we can consider three following cases:
\begin{enumerate}
\item $p_{21}=s_{21}, r_{21}\neq v_{21}$;

Then there exists a unique line 
$$l=l_2=\left\{T_3(q)\left(\left[\begin{array}{cclr}
0&0&0\\ p_{21}&x_{22}&0\\x_{31}&x_{32}&x_{33}\end{array}\right],\left[\begin{array}{cclr}
1&0&0\\0&y_{22}&0\\0&y_{32}&y_{33}\end{array}\right]\right); x_{31}\in F(q)\right\},$$ 
incident with  $P$ and $S$.
\item $p_{21}\neq s_{21}, r_{21}= v_{21}$;

Then there exists a unique line
$$l=l_1=\left\{T_3(q)\left(\left[\begin{array}{cclr}
1&0&0\\ 0&x_{22}&0\\0&x_{32}&x_{33}\end{array}\right],\left[\begin{array}{cclr}
0&0&0\\r_{21}&y_{22}&0\\y_{31}&y_{32}&y_{33}\end{array}\right]\right); y_{31}\in F(q)\right\},$$ 
incident with $P$ and $S$.
\item $p_{21}\neq s_{21}, r_{21}\neq v_{21}$;

Then $l_1$ consists of unimodular free cyclic submodules of $\mathbb{P}(T_3(q))$ generated by pairs
$$\left(\left[\begin{array}{cclr}
1&0&0\\ 0&x_{22}&0\\0&x_{32}&x_{33}\end{array}\right],\left[\begin{array}{cclr}
(p_{21}-s_{21})^{-1}(r_{21}-v_{21})&0&0\\r_{21}-p_{21}(p_{21}-s_{21})^{-1}(r_{21}-v_{21})&y_{22}&0\\
y_{31}&y_{32}&y_{33}\end{array}\right]\right),$$ 
where $y_{31}$ runs through all the elements of $F(q)$, and $l_2$ consists of unimodular free cyclic submodules of $\mathbb{P}(T_3(q))$ generated by pairs
$$\left(\left[\begin{array}{cclr}
(r_{21}-v_{21})^{-1}(p_{21}-s_{21})&0&0\\ p_{21}-r_{21}(r_{21}-v_{21})^{-1}(p_{21}-s_{21})&x_{22}&0\\x_{31}&x_{32}&x_{33}\end{array}\right],\left[\begin{array}{cclr}
1&0&0\\0&y_{22}&0\\0&y_{32}&y_{33}\end{array}\right]\right),$$ 
where $x_{31}$ runs through all the elements of $F(q).$ 

By using Proposition  \ref{nfcs.ufcs.3} again we get that $l_1$ consists of unimodular free cyclic submodules of $\mathbb{P}(T_3(q))$ generated by pairs:
$$\left(\left[\begin{array}{cclr}
1&0&0\\ 0&x_{22}&0\\0&x_{32}&x_{33}\end{array}\right],\left[\begin{array}{cclr}
(p_{21}-s_{21})^{-1}(r_{21}-v_{21})&0&0\\r_{21}-p_{21}(p_{21}-s_{21})^{-1}(r_{21}-v_{21})&y_{22}&0\\
r_{31}-p_{31}(p_{21}-s_{21})^{-1}(r_{21}-v_{21})&y_{32}&y_{33}\end{array}\right]\right),$$ where $p_{31}, r_{31}$ run through all the elements of $F(q)$,
and  $l_2$ consists of unimodular free cyclic submodules of $\mathbb{P}(T_3(q))$ generated by pairs:
$$\left(\left[\begin{array}{cclr}
(r_{21}-v_{21})^{-1}(p_{21}-s_{21})&0&0\\ p_{21}-r_{21}(r_{21}-v_{21})^{-1}(p_{21}-s_{21})&x_{22}&0\\
p_{31}-r_{31}(r_{21}-v_{21})^{-1}(p_{21}-s_{21})&x_{32}&x_{33}\end{array}\right],\left[\begin{array}{cclr}
1&0&0\\0&y_{22}&0\\0&y_{32}&y_{33}\end{array}\right]\right),$$ where $p_{31}, r_{31}$ runs through all the elements of $F(q)$. 

Let $(X_1, Y_1), (X_2, Y_2)$ be pairs generating element of $l_1$ and $l_2$, respectively. Matrix $A=\left[\begin{array}{cclr}
(r_{21}-v_{21})^{-1}(p_{21}-s_{21})&0&0\\p_{21}-r_{21}(r_{21}-v_{21})^{-1}(p_{21}-s_{21})&1&0\\
p_{31}-r_{31}(r_{21}-v_{21})^{-1}(p_{21}-s_{21})&0&1\end{array}\right]$ leads to the equation $A(X_1, Y_1)=(X_2, Y_2)$  for any pair $(X_1, Y_1)$, and  matrix $B=\left[\begin{array}{cclr}
(p_{21}-s_{21})^{-1}(r_{21}-v_{21})&0&0\\r_{21}-p_{21}(p_{21}-s_{21})^{-1}(r_{21}-v_{21})&1&0\\
r_{31}-p_{31}(p_{21}-s_{21})^{-1}(r_{21}-v_{21})&0&1\end{array}\right]$ leads to the equation $B(X_2, Y_2)=(X_1, Y_1)$ for any pair $(X_2, Y_2)$.  
So, by the Remark \ref{generators.of.fcs}, $l_1=l_2$; consequently, there exists a unique line incident with $P$ and $S$.
\end{enumerate}
Thereby axiom  {\bf A1} of an affine plane is proven.

\vspace*{0.3cm}

Let
$$l=\left\{T_3(q)\left(\left[\begin{array}{cclr}
1&0&0\\ 0&x_{22}&0\\0&x_{32}&x_{33}\end{array}\right],\left[\begin{array}{cclr}
y_{11}&0&0\\y_{21}&y_{22}&0\\y_{31}&y_{32}&y_{33}\end{array}\right]\right); y_{31}\in F(q)\right\},$$ 
be a fixed set of parallel lines of an affine plane associated with a subset $C$ of $\mathbb{P}(T_3(q))$. From Proposition \ref{nfcs.ufcs.3} we get:
\begin{itemize}
\item a set of points of any affine plane associated with $C$ such that their elements are not contained in any element of $l$ is of the form $$P=\left\{T_3(q)\left(\left[\begin{array}{cclr}
0&0&0\\ p_{21}&x_{22}&0\\p_{31}&x_{32}&x_{33}\end{array}\right],\left[\begin{array}{cclr}
0&0&0\\r_{21}&y_{22}&0\\r_{31}&y_{32}&y_{33}\end{array}\right]\right); p_{31}, r_{31}\in F(q)\right\},$$ where $r_{21}\neq y_{21}+p_{21}y_{11}$;
\item a set $l'$ of parallel lines of an affine plane associated with $C$ such that any element of $l'$ has nothing in common with any submodule not contained in a non-unimodular free cyclic submodule  and contained in some element of $l$, is of the form
$$l=\left\{T_3(q)\left(\left[\begin{array}{cclr}
1&0&0\\ 0&x_{22}&0\\0&x_{32}&x_{33}\end{array}\right],\left[\begin{array}{cclr}
y_{11}&0&0\\y'_{21}&y_{22}&0\\y_{31}&y_{32}&y_{33}\end{array}\right]\right); y_{31}\in F(q)\right\},$$
where $y'_{21}\neq y_{21}$;
\item
any element of the fixed set $P$ is contained in some element of $l'$ if, and only if, $y'_{21}\neq r_{21}-p_{21}y_{11}$.\end{itemize}
Therefore, for any line $l$ and any point $P$ not on $l$ there is a unique line $l'$ which contains the point $P$ and does not meet the line $l$.
The result is the same in the case of a set $$l=\left\{T_3(q)\left(\left[\begin{array}{cclr}
x_{11}&0&0\\ x_{21}&x_{22}&0\\x_{31}&x_{32}&x_{33}\end{array}\right],\left[\begin{array}{cclr}
1&0&0\\0&y_{22}&0\\0&y_{32}&y_{33}\end{array}\right]\right); x_{31}\in F(q)\right\},$$ 
of parallel lines of an affine plane associated with $\mathbb{P}(T_3(q))$. So, axiom  {\bf A2} of an affine plane is satisfied. 

\vspace*{0.3cm}

Assume that elements of sets 
$$\left\{T_3(q)\left(\left[\begin{array}{cclr}
0&0&0\\ 0&x_{22}&0\\p_{31}&x_{32}&x_{33}\end{array}\right],\left[\begin{array}{cclr}
0&0&0\\0&y_{22}&0\\r_{31}&y_{32}&y_{33}\end{array}\right]\right); p_{31}, r_{31}\in F(q)\right\},$$
$$\left\{T_3(q)\left(\left[\begin{array}{cclr}
0&0&0\\ 0&x_{22}&0\\p_{31}&x_{32}&x_{33}\end{array}\right],\left[\begin{array}{cclr}
0&0&0\\1&y_{22}&0\\r_{31}&y_{32}&y_{33}\end{array}\right]\right); p_{31}, r_{31}\in F(q)\right\}$$
of points of affine planes associated with subset $C$ of $\mathbb{P}(T_3(q))$ are contained in elements of the same set $l$
of parallel lines of an affine plane associated with $C$. By Proposition \ref{nfcs.ufcs.3} we get immediately that
$$l=\left\{T_3(q)\left(\left[\begin{array}{cclr}
0&0&0\\ 0&x_{22}&0\\x_{31}&x_{32}&x_{33}\end{array}\right],\left[\begin{array}{cclr}
1&0&0\\0&y_{22}&0\\0&y_{32}&y_{33}\end{array}\right]\right); x_{31}\in F(q)\right\}$$
and elements of the set 
$$\left\{T_3(q)\left(\left[\begin{array}{cclr}
0&0&0\\ 1&x_{22}&0\\p_{31}&x_{32}&x_{33}\end{array}\right],\left[\begin{array}{cclr}
0&0&0\\0&y_{22}&0\\r_{31}&y_{32}&y_{33}\end{array}\right]\right); p_{31}, r_{31}\in F(q)\right\}$$
of points of an affine plane associated with $C$ are not contained in elements of $l$. It means that  axiom  {\bf A3} of an affine plane is satisfied. 

\vspace*{0.3cm}

Moreover, elements of any set of parallel lines of an affine plane associated with subset $C$ of $\mathbb{P}(T_3(q))$ contain $q$ sets of points of affine planes associated with $C$, what follows directly from Proposition  \ref{nfcs.ufcs.3}. So, there are altogether $q(q+1)^2$ mutually isomorphic $2$-affine planes of order $q$.
\end{proof}
Figure 2 serves as a visualization of the structure of a 2-affine plane of order three. 

\begin{figure}[pth!]
\centering
\includegraphics[width=4cm]{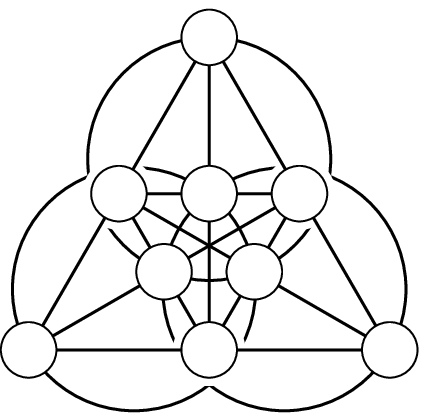} 
\centering
\includegraphics[width=10cm]{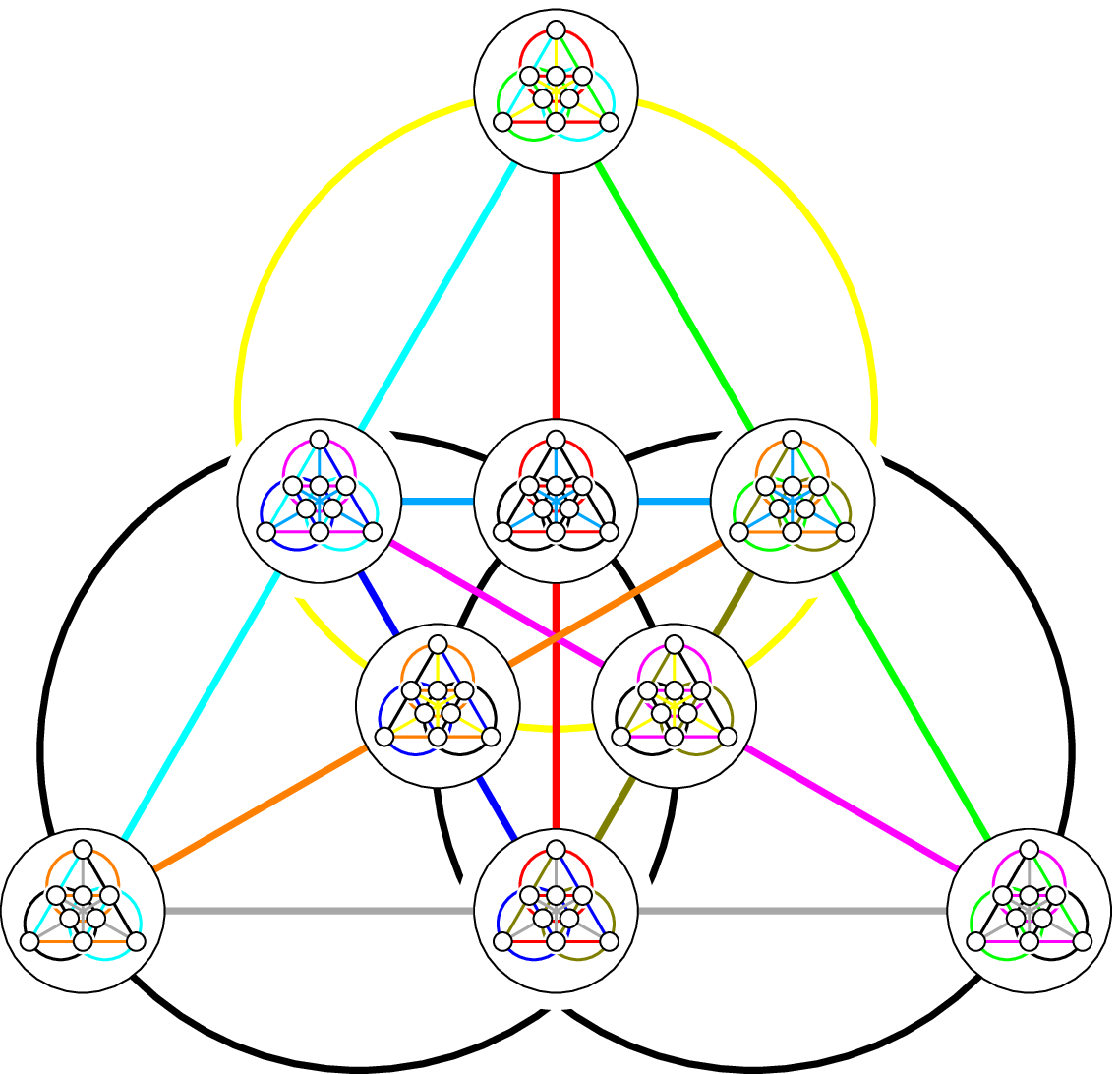}
\caption{A pictorial representation of the affine plane of order three (top) and its 2-affine counterpart (bottom).}
\end{figure}

\section{Concluding remarks}
Although we studied in detail only the case of $n=2$ and $n=3$, it represents no difficulty for the interested reader to readily adjust our lines of reasoning to the case of arbitrary $n$. In the future we also plan to analyse other similar types/classes of finite associative rings with unity, in particular those exhibiting non-unimodular free cyclic submodules whose `homomorphic' images are projective lines themselves.

\section*{Acknowledgment}
This work was supported by both the National Scholarship Programme of the Slovak Republic (E.B.) and the Slovak VEGA Grant Agency, 
Project $\#$ 2/0003/16 (M.S.). We are extremely grateful to Zsolt Szab\'o for an electronic version of Figure 2.

Edyta Bartnicka\\
University of Warmia and Mazury\\
Faculty of Mathematics and Computer Science \\
S\l{}oneczna 54 Street, P-10710 Olsztyn\\ 
Poland\\ 
E-mail: {\tt edytabartnicka@wp.pl}\\ 
\\
Metod Saniga\\
Astronomical Institute\\
Slovak Academy of Sciences\\
SK-05960 Tatransk\' a Lomnica\\ 
Slovak Republic\\
E-mail: {\tt msaniga@astro.sk}


\begin{thebibliography}{10}
\itemsep=-2pt
\bibitem{fcs} Bartnicka, E., Matra{\'s}, A.: 2016, Free Cyclic Submodules in the Context of the Projective Line, Results. Math., Vol. 70, pp. 567--580.
\bibitem{graph} Bartnicka, E., Matra{\'s}, A.: 2017, The distant graph of the projective line over a finite ring with unity, Results. Math., Vol. 72, pp.1943--1958
\bibitem{hart} Hartshorne, R.: 1967, Foundations of Projective Geometry, Benjamin Press. 
\bibitem{design} Havlicek, H.: 2012, Divisible Designs, Laguerre Geometry, and Beyond, J. Math. Sci., New York, 186,  882-926.
\bibitem{havsa}
Havlicek, H., and Saniga, M.: 2008, Projective Ring Line of an Arbitrary Single Qudit, Journal of Physics A: Mathematical and Theoretical, Vol. 41, No. 1, 015302 (12pp). 
\bibitem{hs}
Havlicek, H., and Saniga, M.: 2009, Vectors, Cyclic Submodules and Projective Spaces Linked with Ternions, Journal of Geometry, Vol. 92, Nos. 1-2, pp. 79--90. 
\bibitem{her} Herzer, A.: 1995,  Chain Geometries,  In Buekenhout, F., editor, Handbook of Incidence Geometry, pp. 781--842, Elsevier, Amsterdam.
\bibitem{shpp}
Saniga, M., Havlicek, H., Planat, M., and Pracna, P.: 2008, Twin `Fano-Snowflakes' over the Smallest Ring of Ternions, Symmetry, Integrability and Geometry: Methods and Applications, Vol. 4, Paper 050, 7 pages.
\bibitem{spp}
Saniga, M., Planat, M., and Pracna, P.: 2008, Projective Ring Line Encompassing Two-Qubits, Theoretical and Mathematical Physics, Vol. 155, No. 3, pp. 905--913. 
\bibitem{sp}
Saniga, M., and Pracna, P.: 2010, Space versus Time: Unimodular versus Non-Unimodular Projective Ring Geometries, Journal of Cosmology, Vol. 4, pp. 719--735. 
\end{thebibliography}
\end{document}